\documentclass[10pt,twoside, a4paper, english, reqno]{amsart}
\usepackage[dvips]{epsfig}
\usepackage{amscd}
\usepackage{amssymb}
\usepackage{amsthm}
\usepackage{amsmath}
\usepackage{latexsym}
\usepackage{amsaddr}

\usepackage{upref}
\usepackage{hyperref}

\usepackage{color}
\theoremstyle{plain}
\newtheorem{thm}{Theorem}[section]
\theoremstyle{plain}
\newtheorem{lem}[thm]{Lemma}
\newtheorem{prop}[thm]{Proposition}

\theoremstyle{definition}
\newtheorem{defi}{Definition}[section]
\newtheorem*{rem}{Remark}
\newtheorem*{thmm}{Theorem}

\newcommand{\De} {\Delta}
\newcommand{\la} {\lambda}
\newcommand{\bn}{\mathbb{B}^{2}}
\newcommand{\rn}{\mathbb{R}^{2}}

\newcommand{\authorfootnotes}{\renewcommand\thefootnote{\@fnsymbol\c@footnote}}%

\numberwithin{equation}{section} \allowdisplaybreaks

        \title[]{Existence results for semilinear problems in the two dimensional  
hyperbolic space involving critical growth }

\date{}

\author{Debdip Ganguly${^*}$ \and Debabrata Karmakar$^\dagger$ }

\address{ ${*}$ Dipartimento di Scienze Matematiche, 
Politecnico Di Torino,  Corso Duca degli Abruzzi, 24,
 10129 Torino, Italy.\\
$^\dagger$ Centre for Applicable Mathematics,
    Tata Instiute of Fundamental Research,
     P.O.\ Box 6503, GKVK Post Office,
   Bangalore 560065, India}

\begin{document}

\begin{abstract}
We consider semilinear elliptic problems on two-dimensional hyperbolic space. A model problem of our study is 
\[
  -\De_{g_{\bn}} u = f(x,t), \  u  \in H^{1}(\bn),
\]
 where $H^{1}(\bn)$ denotes the Sobolev space on the disc model of the hyperbolic space and
$f(x,t)$ denotes the function of critical growth in dimension two. 
We first establish the Palais-Smale(P-S) condition for the functional corresponding to the above equation and using (P-S)
condition we obtain existence of solutions. In addition, using concentration argument, we also explore existence
of infinitely many sign changing solutions.
\end{abstract}

\thanks{ AMS subject classifications :  35J20, 35J60 , 35J61, 35B33, 58J05 \\
keywords: Semilinear elliptic problem,  Hyperbolic space,
  Critical growth,  Moser-Trudinger inequality.}

\maketitle

\section{Introduction}

In this article we are concerned with the existence and multiplicity of solutions of the following problem

\begin{equation} \label{main1}
 -\De_{g_{\mathbb{B}^{N}}} u = f(x,u), \  u \in H^{1}(\mathbb{B}^{N}),
\end{equation}
where ${H^{1}}(\mathbb{B}^{N})$ 
denotes the Sobolev
space on the disc model of the hyperbolic space $\mathbb{B}^{N}$ endowed with the Poincar\`{e} metric $g_{\mathbb{B}^{N}},$
 $\De_{g_{\mathbb{B}^{N}}}$  denotes the Laplace Beltrami operator
 on $\mathbb{B}^{N}$ and  $f : {\mathbb{B}^{N}} \times \mathbb{R} \rightarrow \mathbb{R} $ be a $C^{1}$ function 
with $f(x,-t) = -f(x,t).$ \\

Eq. \eqref{main1} has been the subject of intensive research in the past few years after its connection with various 
geometrical problems were discovered. For example, \eqref{main1} with  
 $f(x,t) = \lambda t + |t|^{p-2}t,$  $2 < p \leq \frac{2N}{N-2}$ when $N \geq 3$ and $2 < p < \infty$ when $N = 2,$ 
arises in the study of 
Grushin Operator \cite{B}, Hardy-Sobolev-Maz'ya equation ( \cite{HS1}, \cite{HS}, \cite{MS}) and prescribing
Webster curvature on the Heisenberg group. 
In this case, a great attention has been devoted to the study of positive 
solutions. More precisely existence, uniqueness, regularity, symmetry and  nondegeneracy properties of positive solutions has been 
thoroughly investigated in ( \cite{BGG}, \cite{DS1} and \cite{MS}).

In the seminal paper \cite{MS}, with the above choice of $f$ and $p$ subcritical it has been shown that the problem always 
admits a positive solution. The solutions were also shown to be unique upto hyperbolic isometries except in the case
of dimension two. However when $ N \geq 3$ and $ p = \frac{2N}{N-2},$ i.e., the critical case, the study of existence of solutions
become more interesting due to the lack of compactness of the Sobolev embedding in the hyperbolic space. It has been shown that 
Eq. \eqref{main1} admits a positive solution provided $ \frac{N(N-2)}{4} < \lambda \leq \left( \frac{N-1}{2} \right)^2.$
This is a contrast with the Euclidean case where a positive solution 
 do exist iff $\lambda = 0,$  it is unique upto translations and dilations and is explicitly known.

So the next step is to characterize all sign changing solutions. Existence of sign changing 
solutions has been investigated in (\cite{BGG}, \cite{PS}). Furthermore, extension to general manifolds  
 are also discussed in \cite{EBG}.  The results in \cite{EBG} hold for quite general nonlinearities 
$f$  and non energy solutions are also dealt with. However the critical case $p = \frac{2N}{N-2},$ the problem become 
more delicate and has been thoroughly studied in \cite{DS}. One of the important results obtained in \cite{DS} is the existence
of infinitely many sign changing radial solutions for $N \geq 7.$ So the  question remains open for $N \leq 6.$

In this article, we are interested in the problem \eqref{main1} when $N = 2$ and the nonlinearity is 
``Critical''.  Criticality comes from the critical Sobolev embedding, more precisely Moser-Trudinger inequality (see \cite{JM}). 
First let us recall the Moser-Trudinger (M-T) inequality on the hyperbolic space. 
Recently Mancini-Sandeep in \cite{MS1} and Adimurthi-Tintarev in \cite{ADIT}, proved that
M-T holds true in the hyperbolic space. In fact they showed that:
\begin{thmm} (\cite{MS1}):
 Let $\mathbb{D}$ be the unit open disc in $\mathbb{R}^2,$ endowed with a conformal metric $h = \rho g_{e},$ where 
$g_{e}$ denotes the Euclidean metric and $\rho \in C^{2}(\mathbb{D}), \rho > 0,$ then

 \begin{equation}\label{mtr}
 \sup_{u \in C_{0}^{\infty}(\mathbb{D}), \int_{\mathbb{D}} |\nabla_{h} u|^2  \leq 1} \int_{\mathbb{D}} 
\left (e^{4\pi u^2} - 1 \right) \ dv_{h} < \infty, 
 \end{equation}
holds true if
and only if $h \leq cg_{\bn}$ for some positive constant $c.$
\end{thmm}
The above inequality \eqref{mtr} is sharp, in the sense that the ``critical'' constant $4 \pi$ cannot be improved. We refer 
 \cite{LM}, \cite{MS2} and \cite{TMS}  for Moser-Trudinger inequality in the 
higher dimensional hyperbolic space. However the existence 
of extremals of the above (M-T) inequality is still an \emph{open question}. In this direction some partial results 
were obtained by Manicini-Sandeep-Tintarev \cite{TMS}. They showed the existence of extremals for a modified 
Moser-Trudinger inequality. In particular they proved the following :

 \[
  \tilde S : = \displaystyle{\sup_{||u||_{\mathcal{H}} \leq 1}} \int_{\bn} \left( e^{4\pi u^2} - 1 - 4 \pi u^2 \right) \ dv_{g_{\bn}},  
 \]
is finite and attained or in other words the corresponding Euler Lagarange equation 
\begin{equation}\label{EULER}
 - \Delta_{g_{\bn}} u - \frac{1}{4} u = \beta u (e^{4 \pi u^2} - 1), \ \beta = \frac{1}{\int_{\bn} u^2(e^{4 \pi u^2} - 1)
 \ dv_{g_{\bn}}} 
\end{equation}
admits a positive (radial) solution in $\mathcal{H}$ where $\mathcal{H}$ denotes closure of $C^{\infty}_{0}(\bn)$
with respect to the norm 
\[
 ||u||_{\mathcal{H}}^2 = \int_{\bn} \left [ |\nabla_{g_{\bn}} u|^2 - \frac{1}{4}  |u|^2 \right] \ dv_{g_{\bn}}. 
\]
 Now it is important to remark that the solution of  Eq. \eqref{EULER} $u$ satisfies 
\[
 |u(x)| \geq C (1 - |x|^2)^{\frac{1}{2}},
\]
 and hence not an element of $H^{1}(\bn).$\\

Motivated by the Euler-Lagrange equation \eqref{EULER} satisfied by the Moser-Trudinger inequality,
 we plan to address the question of existence of solutions to problem \eqref{main1} in
dimension two and involving exponential nonlinearity.
In particular, we are interested in the 
 existence of  positive solutions, sign changing solutions and their multiplicity when 
$N = 2$ and $f(x,t) = h(x,t) (e^{\lambda t^2} - 1 )$ is a function of critical growth (see Definition \ref{1}).
 Hence from now onwards we shall consider the following problem 
\begin{equation} \label{main}
 -\De_{g_{\bn}} u = h(x,u) (e^{\lambda u^2} - 1), \  u \in H^{1}(\bn).
\end{equation}

In the Euclidean setting, i.e., 
when Eq. \eqref{main} is posed on  $\Omega \subset \rn,$  a bounded domain, many important existence results were obtained, see for example, 
 Carleson-Chang \cite{CC1}, Atkinson-Peletier \cite{AP}, Adimurthi et al (\cite{ADI}, \cite{ADI1}), 
Panda (\cite{PANDA}, \cite{PANDA1}),
de Figueiredo et al(\cite{DR1}, \cite{DR3}) etc.
Adimurthi \cite{ADI} proved 
existence of non trivial solution and also established Palais-Smale condition for the functional 
corresponding to Eq. \eqref{main}. Thereafter the focus had been on the existence of sign changing solutions. 
In \cite{ADI1}, Adimurthi-Yadava obtained existence of sign changing solution 
when $ {\sup_{x \in \overline{\Omega} }} \
f^{\prime}(x,0) < \mu_{1}(\Omega),$ where $\mu_{1}(\Omega)$ denotes the first eigenvalue 
of Dirichlet boundary value problem involving Euclidean Laplacian.
 In addition, they also proved, when $\Omega$ is a Euclidean ball, 
\eqref{main} admits infinitely many radial sign changing solutions. 
Also in the  critical case Adimurthi-Srikanth-Yadava \cite{AYS} obtained non existence results under some suitable 
conditions for the 
Euclidean setting.
However a complete study of the borderline between existence and non existence has 
been provided  by Adimurthi-Prashanth in \cite{APR}.
All these results uses the variational approach in order to tackle existence results. The key step in using such
a theory is the verification of conditions which allow the use of the Palais-Smale
condition. Recently Guozhen-Nguyen in \cite{LLG}, obtained existence of solutions 
of \eqref{main} without assuming Ambrosetti-Rabinowitz condition.

Before going further, we first introduce the definition of critical growth function.
In view of the Moser-Trudinger embedding in the Euclidean setting, the notion of functions of critical growth was first 
introduced by Adimurthi in \cite{ADI}. However in the same spirit we intend to generalize the concept in the hyperbolic setting.
 The recent development of Moser-Trudinger inequality in the hyperbolic space \cite{MS1}, enables us to define the the following class of critical growth functions.:  
\begin{defi}\label{1}
  Let $h:  \bn \times \mathbb{R} \rightarrow \mathbb{R} $ be a $C^{1}$- function and $\la >0.$ 
The function $f(x,t) = h(x,t) (e^{\la t^2}-1)$   is said to be a function of critical growth on $\bn$ if  $f(x,t) > 0$ for $t > 0,$  $f(x,-t) = -f(x,t)$ and 
 satisfies the following growth conditions  :\\
There exists a constant $M_1 > 0$ such that, for every $\epsilon > 0$ and for all $(x,t) \in  \bn \times (0, \infty),$\\

(C1) \ $h(.,.) \in L^{\infty}(\bn \times [-L,L])$ for all $L > 0,$ and
\[
\sup_{x \in \bn} \ h(x,t) = O(t^a), \ \mbox{near} \ t = 0, \ \mbox{for some} \  a>0. 
\]

(C2) $ f^{\prime}(x,t) > \frac{f(x,t)}{t},$ where  $f^{\prime}(x,t)= \frac{\partial f}{\partial t}(x,t).$ \\

(C3) \ $F(x,t) \leq M_1(g(x) + f(x,t)),$ where $F(x,t) = \int_{0}^{t} f(x,s) ds,$ and $g \in L^{1}(\bn,dv_{g_{\bn}}).$ \\

(C4) \   For any compact set $K \subset \bn,$ it holds,
 \[ 
\lim_{t \rightarrow \infty} \displaystyle{\inf_{ x \in K }} \ h(x,t) e^{\epsilon t^2} = \infty, \ \mbox{and} \
\lim_{t \rightarrow \infty} \displaystyle{\sup_{x \in \bn }} \  h(x,t) e^{-\epsilon t^2} = 0.
\]   

\end{defi}  
For examples of functions of critical growth we refer to section 2. Moreover the class of critical growth functions defined above does not
depend on the choice of the origin, that is,  radiality assumption
can be posed with respect to an arbitrary point in the
hyperbolic space, and changing this point to the origin by M\"{o}bius
transformation will not change the assumptions in Definition \ref{1}.\\

 Now we will briefly discuss some of the hurdles  we may encounter in dealing with the problem in the hyperbolic space.
First of all we have to deal with the infinite volume case which makes the problem very different from the bounded one.
Secondly, one of the major difficulty comes from the lack of compactness. 
 The lack of compactness can occur  due to the concentration phenomenon as well as through the vanishing of mass in the sense of 
the concentration compactness of Lions (see \cite{pl1}).
 In the Euclidean case, by dilating a given sequence we can assume that all the functions involved
has a fixed positive mass in a given ball and hence we can overcome the
vanishing of the mass, but in the case of hyperbolic space $\bn$ this is not possible as the
conformal group of $\bn$ is the same as the isometry group.
 We will overcome this difficulty by using the growth estimates near infinity.\\

To the best of our knowledge, this is the first article which deals with the critical growth function in the two-dimensional hyperbolic
space. We establish Palais-Smale condition for the functional corresponding to \eqref{main} (see Theorem \ref{PS}), which 
led us to the following existence theorem :

\begin{thm}\label{mt1}

Let $f$ be a function of critical growth. Further, assume that $f$ is radial and
for any $K \subset \bn$ compact there holds
\begin{equation}
   \lim_{t \rightarrow \infty } \inf_{x \in K} h(x,t) t = \infty.
\end{equation}

%
Then Eq. \eqref{main} admits a positive solution.
\end{thm}

\begin{rem}
 If we write in the  Euclidean coordinate, then Theorem \ref{mt1} tells us that,
 \begin{equation}
  \Delta u = \left(\frac{2}{1 - |x|^2}\right)^2 h(|x|,u)(e^{\la u^2} - 1)
 \end{equation}
has a radial solution in $H^1_0(\bn),$ under the assumption that $h(x, u) = O(u^a)$ for some $a > 0$ near $``u = 0".$  This
allows us to consider the quadratic singularity (or integrability) at the boundary i.e.,  of order $\frac{1}{(1 - |x|^2)^{2}}.$   
\end{rem}

\begin{rem} 
The above theorem is also true for $f$ non radial with some assumption on the growth of $f.$ Please see section 7, Theorem 
\ref{main theorem for nonradial} for further details. 
\end{rem}

Also using variational methods and concentration argument we obtain the following result :

\begin{thm}\label{mt2}
 Let $f$ be a function of critical growth, radial  and
given any $ N >0$ and compact set $K \subset \bn,$ there exists $t_{N,K} > 0$ such that 
\begin{equation}\label{c4}
\inf_{x \in K} h(x,t)t \geq e^{Nt}, \ \ \forall t \geq t_{N,K} 
 \end{equation}
holds.
Then \eqref{main} has a radial sign changing solution.
\end{thm}
\begin{rem}
 In Theorem \ref{mt2}, condition \eqref{c4} is optimal in order to get a radial sign changing solution. 
If we consider, $f(x,t) = (1 - |x|^2)^2 t e^{t^2 + |t|^{a}}, \ 0 < a \leq 1$ then by conformal invariance, \eqref{main}
 does not admit any radial sign changing solution (see \cite{ADI2}). 
\end{rem}

Once we obtain existence of radial  sign changing solution,
 we can go further to investigate their multiplicity. The main idea is the following : for a positive integer $k,$ one can divide 
$\bn$ into $k$ annuli and considering functions satisfying certain conditions on each annuli one can get existence of solution(s) 
 having k nodes. In precise, we have the following theorem :     
\begin{thm}\label{mt3}
Let $f$ be a function of critical growth, radial and satisfy the condition \eqref{c4}. Then \eqref{main}
has infinitely many radial sign changing solutions.
 
\end{thm}

\begin{rem}
 Theorem \ref{mt3} gives an affirmative answer to the question of existence of infinitely many sign changing radial solutions for 
the problem \eqref{main1} in dimension two. 
\end{rem}

We also give an existence of non radial solution to the above problem. Please see section 7, Theorem \ref{main theorem for nonradial} (Appendix 2) for the discussion and the proof of existence of  non radial solution. \\

The paper is organized as follows. 
We divide the article into seven sections. Sections 2 and 3 discuss the preliminaries and some technical frameworks.
 Section 4 is devoted to the Palais-Smale (P-S) condition and several convergence results.
 The results of Section 4 are used to prove the main existence Theorems
 \ref{mt1}, \ref{mt2}, \ref{mt3} in Section 5. In section 6 we give a sketch of the proof of Lemma \ref{mainlemma} as 
  Appendix 1.  The last section (Appendix 2) is devoted to the existence of non radial solutions.

 \section{Notations and Functional Analytic Preliminaries}

 In this section we will introduce some of the notations and definitions used in this
 paper and also recall some of the embeddings
  related to the Sobolev space in the hyperbolic space. We also obtain estimates for radial functions. \\

We will denote by $\bn$ the disc model of the hyperbolic space, i.e., the unit disc
 equipped with 
 the Riemannian metric $g_{\bn} := \sum\limits_{i=1}^2 \left(\frac{2}{1-|x|^2}\right)^2dx_i^2$. To simplify our notations we will denote $g_{\bn}$
by $g$.\\
 The corresponding volume element is given by $dv_{g} = \big(\frac{2}{1-|x|^2}\big)^2 dx, $ where $dx$ denotes the Lebesgue 
measure on $\rn$.  
 The hyperbolic gradient $\nabla_{g}$ and the hyperbolic Laplacian $\De_{g}$ are
 given by
 \begin{align*}
  \nabla_{g}=\left(\frac{1-|x|^2}{2}\right)^2\nabla,\ \ \ 
 \De_{g}=\left(\frac{1-|x|^2}{2}\right)^2\De \  .
 \end{align*}
{\bf Sobolev Space :} We will denote by ${H^{1}}(\bn)$ the Sobolev space on the disc
 model of the hyperbolic space $\bn.$  \\
Throughout this paper we will denote the norm of $H^{1}(\bn)$ by $||u|| 
: = \left( \int_{\bn} |\nabla_{g} u|^2 \ dv_{g} \right)^{\frac{1}{2}}.$\\
{\bf A sharp Poincar\'{e}-Sobolev inequality :}(see \cite{MS})\\

 For $N \geq 3$ and $p \in \left(1, \frac{N+2}{N-2}      \right]$ there exists an optimal constant 
$S_{N,p} > 0$ such that
\begin{equation}\label{p}
 S_{N,p} \left( \int_{\mathbb{B}^{N}} |u|^{p + 1} d v_{\mathbb{B}^{N}} \right)^{\frac{2}{p + 1}} 
\leq \int_{\mathbb{B}^N} \left[|\nabla_{\mathbb{B}^{N}} u|^{2}
 - \frac{(N-1)^2}{4} u^{2}\right] dv_{\mathbb{B}^{N}},
\end{equation}
for every $u \in C^{\infty}_{0}(\mathbb{B}^{N}).$ If $ N = 2$ any $p > 1$ is allowed.

A basic information is that the bottom of the spectrum of $- \Delta_{g}$ on $\bn$ is 
\begin{equation}\label{firsteigen}
  \frac{1}{4} = \inf_{u \in H^{1}(\bn)\setminus \{ 0 \}} 
\frac{\int_{\bn}|\nabla_{g} u|^2 dv_{g} }{\int_{\bn} |u|^2 dv_{g}}. 
\end{equation}

Also, from the conformal invariance we have,
\begin{lem}\label{1l1}
 If $u \in H^{1}(\bn),$ then 

\begin{equation}\label{conformalgradient}
\int_{\bn} |\nabla_{g} u|^2 dv_{g} = \int_{\bn} |\nabla u|^2 dx, 
\end{equation}
where $\nabla$ denotes the Euclidean gradient on $\rn.$ 
\end{lem}
\begin{proof}
In local Coordinates we have
\begin{align}
 \int_{\bn} |\nabla_{g} u|^2 dv_{g} &= \int_{\bn} 
\left( \frac{1 - |x|^2}{2} \right)^2 |\nabla u|^2 \left( \frac{2}{1 - |x|^2} \right)^2 dx \\ \notag
& = \int_{\bn} |\nabla u|^2 dx.
\end{align}
\end{proof}

%

Let $H^{1}_{R}(\bn)$ denotes the subspace 
\[
 H^{1}_{R}(\bn) := \{ u \in H^{1}_{R}(\bn) : u \ \mbox{is radial} \ \}.
\]
Since the hyperbolic sphere with centre $0 \in \bn$ is also a Euclidean sphere with centre $0 \in \bn$ (see \cite{JR}),
$H^{1}_{R}(\bn)$ can also be seen as the subspace consisting of hyperbolic radial functions.

\begin{prop}\label{eqre}
 Let $u \in H^{1}_{R}(\bn),$ then 
\begin{equation}\label{radialestimate}
|u(x)| \leq \frac{||u||}{ (4 \pi)^{\frac{1}{2}}} \frac{(1- |x|^2)^{\frac{1}{2}}}{|x|^{\frac{1}{2}}}. 
\end{equation}
\end{prop}
\begin{proof}
 Since $u \in H^{1}_{R}(\bn),$ then $u(x) = u(|x|),$ by denoting the radial function by $u$ itself. For u radial,
  in hyperbolic polar co-ordinates $|x| = \tanh \frac{t}{2},$ we have  \\
\[
 \int_{\bn} |\nabla_{g} u|^2 \ dv_{g} = \omega_{2} \int_{0}^{\infty} \sinh t |u^{\prime}(t)|^2 dt  < \infty. 
\]
Thus for $u \in H^{1}_{R}(\bn)$ and $t < \tau,$

\begin{align}
 |u(\tau) - u(t)| & = \left| \int_{t}^{\tau} u^{\prime}(s) ds \right | 
\leq \left ( \int_{0}^{\infty} (\sinh s) |u^{\prime}(s)|^2 ds\right)^{\frac{1}{2}} \times 
\left ( \int_{t}^{\infty} \frac{ds}{\sinh s} \right)^{\frac{1}{2}} \notag \\
& \leq  ||u||_{H^{1}} \left ( \frac{1}{ 2 \pi \sinh t} \right)^{\frac{1}{2}}.
\end{align}
Since $\int_{\bn} u^2 dv_{g} = \omega_{2} \int_{0}^{\infty} u^2 \sinh t dt < \infty,$ 
this implies $\liminf_{\tau \rightarrow  \infty } u(\tau) = 0,$ we get,

\begin{equation}\label{erd1}
 |u(t)| \leq ||u||_{H^{1}} \left( \frac{1}{2 \pi \sinh t} \right)^{\frac{1}{2}}.
\end{equation}
Now substituting $t = 2 \tanh^{-1}(|x|),$
\begin{align}\label{rdh}
 \sinh t & = \frac{e^t - e^{-t}}{2} = \frac{e^{2 \tanh^{-1}(|x|)} - e^{-2 \tanh^{-1}(|x|) }}{2} \notag \\ 
& = \frac{e^{2 \log \left( \frac{1 + |x|}{1 - |x|}\right)} - 1}{ 2 e^{ \log \left( \frac{1 + |x|}{1 - |x|}\right)}} = 
\frac{ 2 |x|}{ (1 - |x|^2)},
\end{align}
hence substituting \eqref{rdh} in \eqref{erd1} we get
\[
 |u(x)| \leq \frac{||u||}{ (4 \pi)^{\frac{1}{2}}} \frac{(1- |x|^2)^{\frac{1}{2}}}{|x|^{\frac{1}{2}}}.
\]
This completes the proof of the proposition.
\end{proof}
 
\begin{rem}
The above proposition is redundant. Instead one can use the standard estimate 
$|u(r)| \leq \frac{1}{(2 \pi)^{1/2}} \sqrt{\log{\frac{1}{r}}}  ||\nabla u||_{2}$ on the ball (See \cite{TINU}) which is sharper  than \eqref{radialestimate} as $r := |x| \rightarrow 1.$  However for the sake notational brevity we use estimate \eqref{radialestimate} and also it does not weakens the results we obtain in this article.  
\end{rem} 

 {\bf Compactness Lemma :} 
%

%
%

Next we shall prove the  compactness lemma of P.L.Lions \cite{PLL} in the hyperbolic setting. 
The main ingredient of the proof is using a suitable covering  of hyperbolic space  with M\"obius transformation developed 
by Adimurthi-Tintarev in their paper \cite{ADIT}.  Adopting their approach we prove the following.  

\begin{lem} [Hyperbolic version of P.L.Lions Lemma] \label{L3.5}
 Let $\{u_k : ||u_k|| = 1\}$ be a sequence in $H^{1}(\bn)$ converging weakly to a non-zero function $u.$ Then
 for every $p < \left(1 - ||u||^2 \right)^{-1},$
 \begin{align} \label{hyperbolic PLL}
  \sup_k \int_{\bn} (e^{4\pi p u^2_k} - 1) \ dv_{g} < \infty.
 \end{align}

\end{lem}

\begin{proof}
 Let us fix an open set $U$ in $\bn$ such that $\overline{U} \subset \bn$ and define :
 \begin{align}
  ||u||_U^2 = \int_U |\nabla u|^2 \ dx + \int_U u^2 \left(\frac{2}{1 - |x|^2}\right)^2 \ dx.
 \end{align}
Then following \cite{ADIT}, we can conclude that , there exists a number $q>0$ such that for all 
$u \in H^1(\bn)$ with $||u||_U < 1,$ there holds
\begin{align}
 \int_U (e^{qu^2} - 1) \ dv_g \leq C \frac{||u||^2_U}{1 - ||u||^2_U}.
\end{align}
Let us fix a $u_k.$
Let $\{\phi_i\}_i$ be a countable family of M\"{o}bius transforms such that $\{\phi_i(U)\}_i$ covers $\bn,$
having finite multiplicity, say $R_0.$ Then define,

\begin{align}
S_k := \left\{i : ||u_k \circ \phi_i||^2_U > \frac{q}{8 \pi p }\right\}. 
\end{align}
Proceeding as in \cite{ADIT} we can show that number of elements in $S_k$ is less than 
$\frac{40 \pi p R_0}{q} + 1,$ and
\begin{align} \label{3.24}
 \sum_{i \notin S_k } \int_{\phi_i(U)} (e^{4\pi p u^2_k} - 1) \ dv_g \leq C,
\end{align}
where $C$ is independent of $u_k.$ Whereas,
\begin{align} \label{3.25}
 \sum_{i \in S_k} \int_{\phi_i(U)} (e^{4\pi p u^2_k} - 1) \ dv_g
 &\leq C \sum_{i \in S_k} \int_{\bn} (e^{4\pi p (u \circ \phi_i)^2_k} - 1) \notag \\
 &\leq C\left(\frac{40 \pi p R_0}{q} + 1 \right),
\end{align}
 by $||v \circ \phi_i|| = ||v|| $ for all $v \in H^1(\bn),$ and the Euclidean version of P.L.Lions lemma \cite{PLL}.
Therefore from \eqref{3.24} and \eqref{3.25} we get  \eqref{hyperbolic PLL}.
\end{proof}

Finally we end this section with some examples of functions having critical growth and definition of Moser functions.\\

{\bf{Examples of functions of critical growth:}} \\

(i) $f(x,t) = t(e^{\la t^2} - 1),$ is an example of function of critical growth. This example suggests we can allow the singularity at the boundary of the ball  of order $\frac{1}{(1 - |x|^2)^{2}}.$

(ii) Let $h(x,t) \in C^1(\bn \times (0,\infty))$ be a positive function satisfying (C1), (C4) and
\[
 h^{\prime}(x,t) \geq \frac{h(x,t)}{t},
\]
then $f(x,t) = h(x,t)(e^{\la t^2} - 1)$ is a function of critical growth. \\
\emph{proof.} One can easily show that $f^{\prime}(x,t) > \frac{f(x,t)}{t}.$ It remains to show that $f$ satisfies (C3). \\
For $t \leq \frac{1}{\sqrt{\la}},$ we have from the definition of $F(x,t) :$
\begin{align*}
 F(x,t) = \int^t_0 f(x,s) \ ds \leq tf(x,t) \leq \frac{1}{\sqrt{\la}} f(x,t).
\end{align*}
 For $t > \frac{1}{\sqrt{\la}},$  we have :
\begin{align*}
 F(x,t)  &= \int^t_{0} h(x,s)(e^{\la s^2} - 1) ds \\
                        &= \frac{1}{2\la}\int^t_{0} \frac{h(x,s)}{s} \frac{d}{ds} (e^{\la s^2} - \la s^2) ds \\
   &= \frac{1}{2\la}\int^t_{0} \frac{1}{s} \left[\frac{h(x,s)}{s} - h^{\prime}(x,s)\right] (e^{\la s^2} - \la s^2) \ ds \\
   & + \frac{1}{2\la}\frac{h(x,t)}{t}(e^{\la t^2} - \la t^2) .
\end{align*}
Therefore using $h^{\prime}(x,t) \geq \frac{h(x,t)}{t},$ we get
\[
 F(x,t) \leq C f(x,t).
\]

This proves $f$ satisfies (C3).\\

{\bf{Definition of Moser Function :}}\\

For $0 < l < R_0 < 1, \ m_{l,R_{0}}(x)$ be the Moser function defined by
 \begin{align*}
  m_{l,R_{0}}(x) = \frac{1}{\sqrt{2 \pi}}
  \begin{cases}
   \left( \log \frac{R_0}{l}\right)^{\frac{1}{2}} \ \mbox{if} \ 0 < |x| < l,  \\
   \frac{\log \frac{R}{|x|}}{( \log \frac{R_0}{l})^{\frac{1}{2}}} \ \ \ \mbox{if} \ l < |x| < R_0,   \\
   0 \ \ \ \ \ \ \ \ \ \ \ \ \mbox{Otherwise},
  \end{cases}
\end{align*}
then $\int_{\bn} |\nabla_{g} m_{l,R_{0}}|^2_{\bn} \ dv_{g} = \int_{\bn} |\nabla m_{l,R_{0}}|^2 \ dx = 1.$


\section{Variational Framework}
We use variational methods in order to prove the main theorems. 
Taking advantage of the Moser-Trudinger inequality and radial estimate \eqref{radialestimate}
 we shall derive a variational principle for \eqref{main} in the sobolev space
$H^{1}_{R}(\bn).$
 The solutions of \eqref{main} are the critical points of the energy functional given by
 
\begin{equation}\label{JF}
J_{\la}(u) = \frac{1}{2} \int_{\bn} |\nabla_{g} u|^2 \ dv_{g} - \int_{\bn} F(x,u) \ dv_{g}.
 \end{equation}

Indeed by Proposition \ref{est1} and Lemma \ref{L3.1}, $J_{\lambda}$ is a well defined $C^1$ functional on $H^{1}_{R}(\bn).$  
Assuming $f$ to be radial in its first variable, it is enough to find critical points of $J_{\la}$ on $H^{1}_{R}(\bn)$
by the principle  of symmetric criticality \cite{PR}.  Hence from now onwards we shall denote $f(x,t) := g(|x|,t)$
by $f$ itself. 

\begin{prop}\label{est1}
If $u \in H^{1}_{R}(\bn),$ then

\begin{equation}
 \int_{\bn} F(x,u) dv_{g} < \infty.
\end{equation}
\end{prop}
\begin{proof}
%
Without loss of generality we can assume that $u \geq 0 .$ By (C2) we have for all $t>0,$
\begin{align} \label{condC3}
 F(x,t) \leq \frac{1}{2} tf(x,t).
\end{align}

Hence using radial estimate \eqref{radialestimate} and \eqref{condC3} we have,
\begin{align}\label{e2}
 \int_{\bn} F(x,u) \ dv_{g} & < \frac{1}{2} \int_{\bn} u f(x,u) \ dv_{g} \notag \\
 & = \int_{\bn \cap \{ |x| > \frac{1}{2}\}} u f(x,u) \ dv_{g} + \int_{ \bn \cap \{|x| < \frac{1}{2}\}} u f(x,u) \ dv_{g}.
\end{align}
Consider the first integral of \eqref{e2},

\begin{align}
 \int_{\bn \cap \{|x| > \frac{1}{2}\}} u f(x,u) \ dv_{g} & = 
\int_{\bn \cap \{|x| > \frac{1}{2}\}} u h(x,u)(e^{\lambda u^2} -1) \ dv_{g} \notag \\
& = \int_{\bn \cap \{|x| > \frac{1}{2}\}} u h(x,u) \frac{(e^{\lambda u^2} -1)}{(1 - |x|^2)^2} \ dx \notag \\
& \leq C ||u||^{\frac{3}{2}} \int_{\bn \cap \{|x| > \frac{1}{2}\}} (1 - |x|^2)^{- \frac{1}{2}} \ dx < \infty. 
\end{align}
The second integral of \eqref{e2} is finite by using Euclidean version of 
Moser-Trudinger inequality \eqref{mtr}. Hence this proves the proposition.  
\end{proof}


Before going further we need some notations and definitions. Let $f$ be a function of critical growth on $\bn.$ Define 

\begin{align}
 \mathcal{M} & = \left \{ u \in H^{1}_{R}(\bn) \setminus \{0 \}: \ ||u||^2 = \int_{\bn} f(x,u)u \ dv_{g} \right \}, \notag \\
 & \mathcal{M}_{1} = \left \{ u \in \mathcal{M} : u^{\pm} \in \mathcal{M}  \right \}, \notag \\
& I_{\la}(u) = \frac{1}{2} \int_{\bn} f(x,u)u \ dv_{g} - \int_{\bn} F(x,u) \ dv_{g}, \notag 
\end{align}
\[
 \frac{\eta (f)^2}{2} = \inf_{u \in \mathcal{M}} J_{\la}(u), 
\]
\[
 \frac{\eta_{1} (f)^2}{2} = \inf_{u \in \mathcal{M}_{1}} I_{\la}(u).
\]

We show the existence of solutions of Eq. \eqref{main} by minimizing the functional $J_{\lambda}$ over $\mathcal{M}.$ However
the main difficulty lies in the validity of Palais-Smale condition. The next section is devoted to the study of 
Palais-Smale condition.

\begin{section}{ Palais-Smale condition and some convergence results}

In this section we study the Palais-Smale condition of the following problem

\begin{equation}\label{E:1.1}
\left.
 \begin{array}{rlll}

  -\Delta_{g}  u  & =  f(x,u) &{\rm in} \; \bn , \\
       u & \in H^{1}(\bn), 
 \end{array} \right\}
\end{equation}
where $f(x,u)$ denotes the function of critical growth. We say $u_{k} \in H^{1}(\bn)$ is a Palais-Smale sequence for $J_{\lambda}$
at a level $c$ if $J_{\lambda}(u_{k}) \rightarrow c $ and $J^{\prime}_{\lambda}(u_{k}) \rightarrow 0$ in $H^{-1}(\bn).$
We show that if we restrict $J_{\lambda}$ to $H^{1}_{R}(\bn),$ then $J_{\lambda}$ satisfy the (P-S)$_{c}$ condition 
for all $c \in \left(0,\frac{2 \pi }{\lambda} \right).$ To be precise we state the following theorem :

\begin{thm}\label{PS}
 Let $f(x,t) = h(x,t)(e^{\la t^2} - 1)$ be a function of critical growth on $\bn$ and $J_{\lambda} : H^1_R(\bn) \rightarrow \mathbb{R}$
 be defined as in \eqref{JF}. Then \\

\item[(i)] $J_{\lambda}$ satisfies Palais-Smale condition on $(0, \frac{2\pi}{\la});$
 \item[(ii)] Moreover if $h$ satisfies,
 \begin{equation} \label{PSlavel}
  \displaystyle{\overline{\lim_{t \rightarrow \infty}} \inf_{x \in K} h(x,t)t = \infty} \ \mbox{for any compact subset} \ 
 K \ \mbox{of} \ \bn,
\end{equation}
then
\begin{align*}
0< \eta(f)^2 <  \frac{4 \pi}{\la}.
\end{align*}
\end{thm}

Above theorem will play a crucial role in the study of existence of solutions. The main difficulties for studying Palais-Smale 
condition is coming from the concentration phenomenon and through vanishing of mass. However vanishing can be handled by using
the radial estimate proved in section 2, Lemma \ref{radialestimate}.
 Keeping this in mind, we plan to address some of the important propositions involving 
convergence of critical growth functions. The propositions and lemmas 
needed in the proof of Theorem \ref{PS} are collected below.

\begin{lem} \label{L3.1}
Let $f(x,t) = h(x,t)(e^{\la t^2} - 1)$ be a function of critical growth. Then we have \\

\item[(i)] $f(x,u) \in L^p(\bn ,dv_{g}),$ for all $p \in [1, \infty) \ \mbox{and} \ u \in H^1(\bn).$  \\

\item[(ii)] $I_{\lambda}(u) \geq 0$ for all $u$ and $I_{\lambda}(u) = 0$ if and only if $u \equiv 0.$  
Moreover, given $\epsilon > 0,$ there exists a constant $C_0(\epsilon) > 0$ such that for all $u \in H^1_R(\bn),$
 \begin{equation} \label{3.0}
 \int_{\bn} f(x,u)u \ dv_{g} \leq C_0(\epsilon)(1 + I_{\lambda}(u)) + \epsilon ||u||^2 .
 \end{equation}
\end{lem}

 	\begin{proof}
 \item[(i)] By (C4) for a given $\epsilon > 0,$ there exists an $N_0 > 0$ such that for all $t \geq N_0$ we have
 \begin{align*}
  f(x,t) \leq C(e^{(\la + \epsilon)t^2} - 1).
 \end{align*}

 For $p \in [1, \infty),$ using the inequality 
$(e^{t} - 1)^{p} \leq  (e^{pt} - 1)$ for $ t \geq 0$ and  hyperbolic version of Moser-Trudinger inequality \eqref{mtr} 
we have

 \begin{align*}
  \int_{\bn} |f(x,u)|^p \ dv_{g} &\leq \int_{\{|u| > N_0\}} |f(x,u)|^p \ dv_{g} + \int_{\{|u| \leq N_0\}} |f(x,u)|^p \ dv_{g} \\
                         &\leq C\int_{\bn} (e^{p(\la + \epsilon)u^2} - 1) \ dv_{g} + C\int_{\bn} (e^{\la p u^2} - 1) \ dv_{g}, \\
                         &< +\infty.
 \end{align*}
 \item[(ii)] By (C2), $f(x,t)t - 2F(x,t) \geq 0$ and equal to $0$ iff $t = 0,$ and hence this implies 
\[
 I_{\lambda}(u) \geq 0 \ \mbox{and} \   I_{\lambda}(u) = 0
 \ \mbox{if and only if} \ u \equiv 0.
\]

  For the second part it is enough to prove the inequality for all 
 $u \in H^1_R(\bn)$ with $u \geq 0.$ Fix $\epsilon >0,$ by (C3), $F(x,t) \leq M_1(g(x) + f(x,t))$ for some positive function 
 $g \in L^1(\bn ,dv_{g}).$ Then
\begin{align} \label{3.1}
  2I_{\lambda}(u) &= \int_{\bn} [f(x,u)u - 2F(x,u)] \ dv_{g} \notag \\
        &\geq \int_{\bn} [f(x,u)u - 2M_1(g(x) + f(x,t))] \ dv_{g} \notag \\
        &= \int_{\bn} f(x,u)(u - 2M_1) \ dv_{g} -2M_1 \int_{\bn} g(x) \ dv_{g}, \notag \\
        &\geq \int_{\bn} f(x,u)(u - 2M_1) \ dv_{g} - C.
\end{align}
Observing that  $ u - 2 M_{1} \geq \frac{1}{2} u$ on $\{ u \geq 4 M_{1} \},$ we have
\begin{align} \label{3.2}
 \int_{\bn} f(x,u)(u - 2M_1) \ dv_{g} &= \int_{\bn \cap \{u \leq 4M_1\}} f(x,u)(u - 2M_1) \ dv_{g} \notag\\ 
                                                 & + \int_{\bn \cap \{u > 4M_1\}} f(x,u)(u - 2M_1) \ dv_{g}  \notag \\
                                    & \geq  \int_{\bn \cap \{u \leq 4M_1\}} f(x,u)(u - 2M_1) \ dv_{g}, \notag \\ 
                                                 & + C\int_{\bn \cap \{u > 4M_{1}\}} f(x,u)u \ dv_{g}.
\end{align}
Therefore from \eqref{3.1} and \eqref{3.2} we have
\begin{align} \label{3.1.1}
 \int_{\bn \cap \{u > 4M_{1}\}} f(x,u)u \ dv_{g} &\leq  C \left |\int_{\bn \cap \{u \leq 4M_1\}} f(x,u)(u - 2M_1) \ dv_{g} \right|
 \notag \\
                                                 &+  C\int_{\bn} f(x,u)(u - 2M_1) \ dv_{g}, \notag \\
                                                 & \leq C \left |\int_{\bn \cap \{u \leq 4M_1\}} f(x,u)(u - 2M_1) 
\ dv_{g} \right| + C(1 + I_{\lambda}(u)).
\end{align}

 Next we estimate $\left |\int_{\bn \cap \{u \leq 4M_1\}} f(x,u)(u - 2M_1) \ dv_{g} \right|.$
 
 Let $\delta >0$ be a small number depending on $\epsilon$ and whose smallness will be decided later.
 By radial estimate \eqref{radialestimate}, there exists a compact set $K_0$ such that the set 
 $\{u > \delta\}$ is contained in $K_0,$ for every $u \in H^1_R(\bn).$
 We can write
 \begin{align} \label{3.3}
  \int_{\{u \leq 4M_1\}} f(x,u) \ dv_g 
  &= \int_{\{\delta < u \leq 4M_1\}} h(x,u)(e^{\la u^2} - 1) \ dv_g 
  + \int_{\{u \leq \delta\}} h(x,u)(e^{\la u^2} - 1) \ dv_g, \notag \\
 &\leq C + \int_{\{u \leq \delta\}} h(x,u)(e^{\la u^2} - 1) \ dv_g.
 \end{align}
Now by (C1), we can estimate the last integral in \eqref{3.3} as
\begin{align} \label{3.3.1}
 \int_{\{u \leq \delta\}} h(x,u)(e^{\la u^2} - 1) \ dv_g
 &\leq C \int_{\{u \leq \delta\}} u^a (e^{\la u^2} - 1) \ dv_g \notag \\
 &\leq C \delta^a \int_{\bn} u^2 \ dv_g \notag \\
 &\leq C \delta^a ||u||^2,
 \end{align}
where the constant $C$ in \eqref{3.3.1} does not depends on $u.$ Now choosing $6M_1 C \delta^a < \frac{\epsilon}{2}$ 
we get
\begin{align} \label{3.3.2}
 \left |\int_{\bn \cap \{u \leq 4M_1\}} f(x,u)(u - 2M_1) \ dv_{g} \right| 
 \leq C + \frac{\epsilon}{2}||u||^2.
\end{align}

Similarly it follows that 
\begin{equation} \label{3.4}
\int_{\bn \cap \{u \leq 4M_{1} \}} f(x,u)u \ dv_{g} \leq C + \frac{\epsilon}{2}||u||^2. 
\end{equation}
Hence from \eqref{3.1.1}, \eqref{3.3.2} and \eqref{3.4} we get
\begin{align*}
 \int_{\bn} f(x,u)u \ dv_g &= \int_{\bn \cap \{u \leq 4M_{1} \}} f(x,u)u \ dv_{g}
                             + \int_{\bn \cap \{u > 4M_{1}\}} f(x,u)u \ dv_{g} \\
                           & \leq C_0(1 + I_{\lambda}(u)) + \epsilon ||u||^2.
\end{align*}
\end{proof}

\begin{lem} \label{L3.2}
Let $f(x,t) = h(x,t)(e^{\la t^2} - 1)$ be function of critical growth, then
\begin{align*}
 \tilde c^2 := \sup \{c^2 : \sup_{u \in H^{1}_{R}(\bn), ||u|| \leq 1} \int_{\bn} f(x,cu)u \ dv_{g} < +\infty \} = \frac{4\pi}{\la}.
\end{align*}

\end{lem}

\begin{proof}
 Fix $\alpha \in (0,1)$ and $\epsilon >0$, by (C4), there exists constants $t_1 , t_2, C_1(\epsilon), C_2(\epsilon) > 0$
 such that
 \begin{equation} \label{3.5}
  f(x,t)t \leq C_1(\epsilon)(e^{\la (1 + \epsilon) t^2} - 1), \quad \mbox{for all} ~ t \geq t_1,
 \end{equation}
 \begin{equation} \label{3.6}
  f(x,t)t \geq C_2(\epsilon)(e^{\la (1 - \epsilon) t^2} - 1), \quad \mbox{for all} ~ t \geq t_2 ~ \mbox{and} ~|x| \leq \alpha .
 \end{equation}
 Now assume $c>0$ be such that :  $\displaystyle{\sup_{u \in H^{1}_{R}(\bn), ||u|| \leq 1} \int_{\bn} f(x,cu)u \ dv_{g} < +\infty} .$ 
 Then using 
 \eqref{3.6}
 \begin{align} \label{3.7}
   \int_{\bn} f(x,cu)u \ dv_{g} &= \frac{1}{c}\int_{\bn} f(x,cu)(cu) \ dv_{g}, \notag \\
                            &\geq \frac{1}{c} \int_{\{|x| \leq \alpha\} \cap \{u \geq \frac{t_2}{c}\}} f(x,cu)(cu) \ dv_{g}, \notag \\
                            &\geq \frac{C_2(\epsilon)}{c} \int_{\{|x| \leq \alpha\} \cap \{u \geq \frac{t_2}{c}\}} [e^{\la (1 - \epsilon)c^2 u^2} - 1] \ dv_{g}, 
  \end{align}
and 
\begin{equation} \label{3.8}
 \int_{\{|x| \leq \alpha\} \cap \{u \leq \frac{t_2}{c}\}} [e^{\la (1 - \epsilon)c^2 u^2} - 1] \ dv_{g} \leq C(\alpha, t_2, c).
\end{equation}
Therefore \eqref{3.7} and \eqref{3.8} together gives
\begin{equation} \label{3.9}
 \int_{\{|x| \leq \alpha\}} (e^{\la (1 - \epsilon)c^2 u^2} - 1) \ dv_{g} 
\leq C(\alpha,\epsilon, c)\int_{\bn} f(x,cu)u \ dv_{g} + C(\alpha, t_2, c).
\end{equation}
 Define $\tilde C = \frac{1}{\sqrt{4\pi}} \left[\frac{1 - \alpha ^2}{\alpha}\right]^{\frac{1}{2}}.$ 
Now by radial estimate we have : for all $u \in H^{1}_{R}(\bn)$ with $||u|| \leq 1,$
\begin{align*}
 |u(x)| \leq \tilde C \quad \mbox{whenever}~ |x| > \alpha.
\end{align*}

Therefore we have
\begin{align} \label{3.10}
 \int_{\{|x| > \alpha\}} (e^{\la (1 - \epsilon)c^2 u^2} - 1) \ dv_{g} 
 &\leq \int_{\{|u| \leq \tilde C\}} (e^{\la (1 - \epsilon)c^2 u^2} - 1) \ dv_{g} \notag \\
&\leq C\int_{\{|u| \leq \tilde C\}} u^2 e^{\la (1 - \epsilon)c^2 u^2} \ dv_{g} \notag \\
                                                              &\leq C\int_{\bn} u^2 \ dv_{g} \notag \\
                                                              &\leq C||u||^2, \notag \\ 
                                                              &\leq C.
\end{align}
Taking into account \eqref{3.9} and \eqref{3.10} we obtain
\begin{equation}
 \sup_{u \in H^{1}_{R}(\bn), ||u|| \leq 1} \int_{\bn} (e^{\la (1 - \epsilon)c^2 u^2} - 1) \ dv_{g} < +\infty.
\end{equation}
Now using hyperbolic version of Moser-Trudinger inequality \eqref{mtr},  
 we have $(1 - \epsilon)c^2 \leq \frac{4\pi}{\la}.$ Since $\epsilon > 0$ was arbitrary, we deduce that 
$\tilde c^2 \leq \frac{4\pi}{\la}.$ \\

Now suppose $\tilde c^2 < \frac{4\pi}{\la}.$ Choose $\epsilon > 0$ such that $(1 + \epsilon)^3 \tilde c^2 < \frac{4\pi}{\la}.$
Then for all $u \in H^{1}_{R}(\bn)$ with $||u|| \leq 1$ we have,
\begin{align} \label{3.11}
 \int_{\bn} f(x,(1 + \epsilon)\tilde cu)u \ dv_{g} 
 &\leq C\int_{\{|u| > \frac{t_1}{(1 + \epsilon)\tilde c}\}} f(x,(1 + \epsilon)\tilde cu)(1 + \epsilon)\tilde c u \ dv_{g} \notag \\
 & + C\int_{\{|u| \leq \frac{t_1}{(1 + \epsilon)\tilde c}\}} f(x,(1 + \epsilon)\tilde cu)(1 + \epsilon)\tilde c u \ dv_{g}, \notag \\
 &\leq C\int_{\bn} (e^{\la (1 + \epsilon)^3 \tilde c^2 u^2} - 1) \ dv_{g} \notag \\ 
       & + C\int_{\{|u| \leq \frac{t_1}{(1 + \epsilon)\tilde c}\}} (e^{\la (1 + \epsilon)^2 \tilde c^2 u^2} - 1) \ dv_{g}, \notag \\
 &\leq C + C\int_{\{|u| \leq \frac{t_1}{(1 + \epsilon)\tilde c}\}} u^2 e^{\la (1 + \epsilon)^2 \tilde c^2 u^2} \ dv_{g} \notag \\
 &\leq C + C\int_{\bn} u^2 \ dv_{g}, \notag \\
 &\leq C + C||u||^2.
\end{align}                                          
As a consequence we derive that :
\begin{align*}
 \sup_{u \in H^{1}_{R}(\bn), ||u|| \leq 1} \int_{\bn} f(x,(1 + \epsilon)\tilde c u)u \ dv_{g} < +\infty,
\end{align*}

which contradicts the definition of $\tilde c.$ So we must have $\tilde c^2 = \frac{4\pi}{\la}.$
\end{proof}

\begin{prop} \label{L3.3}
 Let $\{u_k\}$ be a sequence in $H^{1}_{R}(\bn)$ such that $u_k$ converges weakly to some $u$ in the space $H^{1}_{R}(\bn)$ and assume that
 \begin{align*}
  \sup_k \int_{\bn} f(x,u_k)u_k \ dv_{g} < +\infty.
 \end{align*}

 Then we have the following convergence results :
 
  \item[(i)] $\displaystyle{\lim_{k \rightarrow \infty} \int_{\{|x| < \alpha \} } f(x,|u_k|) \ dv_{g} =
                                                      \int_{\{|x| < \alpha \} } f(x,|u|) \ dv_{g}},$
                                                      for any $\alpha < 1.$
  \item[(ii)] $\displaystyle{\lim_{k \rightarrow \infty} \int_{\bn} F(x,u_k) \ dv_{g} = \int_{\bn} F(x,u) \ dv_{g}}.$
 
\end{prop}

\begin{proof}
 \item[(i)] Fix $\alpha >0,$ then we have  
\[
 \int_{\{|x| < \alpha \} \cap \{|u_k| > N\}} f(x,|u_k|) \ dv_{g} \leq 
                        \frac{1}{N} \int_{\bn} f(x,|u_k|)|u_k| dv_{g} \leq \frac{C}{N}.
\]

Therefore
\begin{align*}
  \int_{\{|x| < \alpha \}} f(x,|u_k|) \ dv_{g} &= \int_{\{|x| < \alpha \} \cap \{|u_k| \leq N\}} f(x,|u_k|) \ dv_{g} 
                                                  + \int_{\{|x| < \alpha \} \cap \{|u_k| > N\}} f(x,|u_k|) \ dv_{g} \\
                                           &=\int_{\{|x| < \alpha \} \cap \{|u_k| \leq N\}} f(x,|u_k|) \ dv_{g} + O\Big(\frac{1}{N}\Big).
\end{align*}
Hence using dominated convergence theorem, letting $k \rightarrow \infty$ followed by $ N \rightarrow \infty$ we have

\begin{equation}\label{p11}
 \lim_{k \rightarrow \infty} \int_{\{|x| < \alpha \}} f(x,|u_k|) \ dv_{g} = \int_{\{|x| < \alpha \}} f(x,|u|) \ dv_{g}. 
\end{equation}


\item[(ii)] Fix some $\alpha \in (0,1)$ close to $1.$ Since $u_k \rightharpoonup u$ in $H^{1}_{R}(\bn)$ we have 
$\displaystyle{\sup_k} \  ||u_k|| \leq C < +\infty.$ So by radial estimate \eqref{radialestimate}, 
\begin{align*}
 \sup_k \ |u_k(x)| \leq C(1 - |x|^2)^{\frac{1}{2}} \ \mbox{for} \ |x| > \alpha.
\end{align*}

So that
\begin{equation} \label{3.12}
 F(x,u_k) \leq C(1 - |x|^2)^{\frac{3}{2}} \ \mbox{for} \ |x| > \alpha.
\end{equation}
Since $(1 - |x|^2)^{\frac{3}{2}} \in L^1(\{|x| > \alpha\}, dv_{g})$ by dominated convergence theorem we get
\begin{equation} \label{3.13}
 \lim_{k \rightarrow +\infty} \int_{\{|x| > \alpha\}} F(x,u_k) \ dv_{g} = \int_{\{|x| > \alpha\}} F(x,u) \ dv_{g}. 
\end{equation}
For $\{|x| < \alpha\},$ we can use (C3) and \eqref{p11} to conclude
\begin{equation} \label{3.14}
 \lim_{k \rightarrow +\infty} \int_{\{|x| < \alpha\}} F(x,u_k) \ dv_{g} = \int_{\{|x| < \alpha\}} F(x,u) \ dv_{g}. 
\end{equation}
So \eqref{3.13} and \eqref{3.14} together proves the lemma. 
\end{proof}


\begin{prop} \label{L3.4}
 Let $\{u_k\}$ and $\{v_k\}$ be bounded sequences in $H^{1}_{R}(\bn)$ converging weakly  to 
 $u$ and $v$ respectively. Further assume that 
 \begin{align*}
  \sup_k \ ||u_k||^2 < \frac{4\pi}{\la}.
 \end{align*}
 
 then for every $l \geq 2,$
 \begin{equation} \label{3.15}
  \lim_{k \rightarrow +\infty} \int_{\bn} \frac{f(x,u_k)}{u_k} v^l_k \ dv_{g} = \int_{\bn} \frac{f(x,u)}{u} v^l \ dv_{g}. 
 \end{equation}
 \end{prop}
 
\begin{proof}
 Fix $\delta > 0,$ since $v_k$ converges weakly to $v$ we have, 
 $\displaystyle{\sup_k} \ ||v_k|| < +\infty.$ Let $\tilde C^2 > \frac{4\pi}{\la}$ be such that
 \begin{align*}
  \sup_k  \ ||v_k|| \leq \tilde C.
 \end{align*}

 Define $\alpha = \frac{1}{\tilde C}\left[\sqrt{\tilde C^2 + 
\big(\frac{2\pi \delta^2}{\tilde C}\big)^2} - \frac{2\pi \delta^2}{\tilde C} \right],$
then by radial estimate it holds 
\begin{equation} \label{3.16}
\sup_k |u_k(x)| \leq \delta, \ \mbox{whenever} \ |x| > \alpha, 
\end{equation}
and 
\begin{equation} \label{3.17}
 \sup_k \ |v_k(x)| \leq \frac{\tilde C}{\sqrt{4\pi \alpha}} (1 - |x|^2)^{\frac{1}{2}}, \ \mbox{whenever} |x| > \alpha.
\end{equation}
Now since $\displaystyle{\sup_k} \ 
 ||u_k||^2 < \frac{4\pi}{\la},$ we can choose $\epsilon > 0$ sufficiently small and $p > 1$ such that
\begin{equation} \label{3.18}
 (\la + \epsilon)p ||u_k||^2 < 4\pi, \ \mbox{for all} \ k.
\end{equation}
By (C4), there exists $N_0 > 0$ such that for all $t \geq N_0,$ 
\begin{align*}
 h(x,t) \leq Ce^{\epsilon t^2}, \ \mbox{for all} \ x.
\end{align*}

Therefore for all $N \geq N_0,$ it holds
\begin{align} \label{3.19}
 \int_{\{|u_k| > N\}} |f(x,u_k)|^p \ dv_{g} &= \int_{\{|u_k| > N\}} |h(x,u_k)|^p (e^{\la u^2_k} - 1)^p \ dv_{g} \notag \\
                                        &\leq C\int_{\{|u_k| > N\}} e^{\epsilon p u^2_k}(e^{\la p u^2_k} - 1) \ dv_{g} \notag \\
                                        &\leq C\int_{\{|u_k| > N\}} (e^{(\la + \epsilon)p u^2_k} - 1) \ dv_{g} \notag \\
                                        &\leq C\int_{\bn} (e^{(\la + \epsilon)p||u_k||^2 \left(\frac{u_k}{||u_k||}\right)^2} - 1) \ dv_{g}, \notag \\
                                        &\leq C_1.
\end{align}
Let $q$ be the conjugate exponent of $p$, then
\begin{equation} \label{3.20}
\sup_k \int_{\bn} |v_k|^{lq} \ dv_{g} \leq C \sup_k ||v_k||^{lq} \leq C(l,q). 
\end{equation}
By \eqref{3.19}, \eqref{3.20} and H\"older's inequality we have
\begin{equation} \label{3.21}
 \int_{\{|u_k| > N \}} \frac{f(x,u_k)}{u_k} v^l_k \ dv_{g} = O\left(\frac{1}{N}\right).
\end{equation}
Now using \eqref{3.16} and \eqref{3.17} we get
\begin{align} \label{3.22}
 |\int_{\{|u_k| \leq \delta \}} \frac{f(x,u_k)}{u_k}v^l_k \ dv_{g}| 
                 &\leq \int_{\{|u_k| \leq \delta \}} |h(x,u_k)|\frac{(e^{\la u^2_k} - 1)}{|u_k|} |v^l_k| \ dv_{g} \notag \\
                 &\leq \int_{\{|u_k| \leq \delta \}} |h(x,u_k)||u_k|\frac{(e^{\la u^2_k} - 1)}{|u^2_k|} |v^l_k| \ dv_{g} \notag \\
                 &\leq C\int_{\{|u_k| \leq \delta \}} |u_k| |v^l_k| \ dv_{g} \notag \\
                 &\leq C\int_{\{|x| > \alpha \}} (1 - |x|^2)^{\frac{1 + l}{2}} \ dv_{g} + O(\delta), \notag \\
                 &= \circ(1) \ \mbox{as} \ \delta \rightarrow 0.
\end{align}
From \eqref{3.21}, \eqref{3.22} we get
\begin{equation} \label{3.23}
 \int_{\bn} \frac{f(x,u_k)}{u_k} v^l_k \ dv_{g} = \int_{\{\delta \leq |u_k| \leq N\}} \frac{f(x,u_k)}{u_k} v^l_k \ dv_{g}
                                                   + O\left(\frac{1}{N}\right) + O(\delta).
\end{equation}
Now the proof follows by dominated convergence theorem and thereafter tending $N \rightarrow \infty, \delta
\rightarrow 0.$
\end{proof}

\begin{rem}
 By taking $v_{k} = u_{k}$ and $l = 2,$ we see that, if $\displaystyle{\sup_{k}} \ ||u_{k}||^2 < \frac{4 \pi}{\lambda}$ then
\[
 \int_{\bn} f(x,u_{k}) u_{k} \ dv_{g} \rightarrow \int_{\bn} f(x,u) u \ dv_{g}.
\]
In general it is difficult to prove $\displaystyle{\sup_{k}} \ ||u_{k}||^2 < \frac{4 \pi}{\lambda}$ from the functional itself. 
However we need this compactness criterion in order to get existence of minimizer on $\mathcal{M}$ and hence 
a solution of Eq. \eqref{main}.
\end{rem}
Next we will investigate under which circumstances we can pass the limit without the condition mentioned in the above remark. 
We use the  hyperbolic version of  P.L. Lions Lemma  \ref{L3.5} in 
order to give an affirmative answer on passing the limit.

\begin{prop} \label{L3.6}
Let $\{u_k\}$ be a sequence converging weakly to a non-zero
function u in $H^{1}_{R}(\bn)$ , and assume that :\\

\item[(i)] there exists $c \in (0, \frac{2\pi}{\la})$ such that $J_{\lambda}(u_k) \rightarrow c.$\\

\item[(ii)] $||u||^2 \geq \int_{\bn} f(x,u)u \ dv_{g}.$\\

\item[(iii)] $\displaystyle{\sup_k} \ \int_{\bn} f(x, u_k)u_k \ dv_{g} < +\infty.$\\

Then
\begin{align*}
 \lim_{k \rightarrow \infty} \int_{\bn} f(x, u_k)u_k \ dv_{g} = \int_{\bn} f(x, u)u \ dv_{g}.
\end{align*}

\end{prop}

\begin{proof}

 Arguing as in the proof of ( \cite{ADI}, Lemma 3.3 ) and using hyperbolic version of P.L.Lions Lemma \ref{L3.5}, we
get
 \begin{align*}
  \sup_k \int_{\bn} (e^{(1 + \epsilon)\la u^2_k} - 1) \ dv_{g} < +\infty.
 \end{align*}

 By (C4), we can assume 
 \begin{align*}
  M_2 = \sup h(x,t)t e^{-\frac{1}{2}\la \epsilon t^2} < +\infty,
 \end{align*}
 so that
 \begin{align} \label{3.26}
  \int_{\{|u_k| > N\}} f(x,u_k)u_k \ dv_{g} &= \int_{\{|u_k| > N\}} h(x,u_k)u_k (e^{\la
u^2_k} - 1) \ dv_{g} \notag \\
                                    & \leq  \int_{\{|u_k| > N\}} (h(x,u_k)u_k e^{-\la
\epsilon u^2_k})(e^{(1 + \epsilon) \la u^2_k} -
1) \ dv_{g} \notag \\
                                    &\leq M_2 e^{-\frac{1}{2} \la \epsilon N^2}
\int_{\bn}  (e^{(1 + \epsilon)\la u^2_k} - 1) \
dv_{g} \notag \\
                                    &\leq C e^{-\frac{1}{2} \la \epsilon N^2},
 \end{align}
 holds.\\
Fix $\delta > 0$ and let $\tilde C$ be such that $\displaystyle{\sup_k} \  ||u_k|| \leq \tilde C$ and
let $\alpha$ depending on $\tilde C$ be as before. Then 
$\alpha = 1 - O(\delta)$ as $\delta \rightarrow 0,$ and there holds
\begin{align*}
 |u_k(x)| \leq  C (1 - |x|^2)^{\frac{1}{2}}, \ \mbox{whenever} \ |x| > \alpha.
\end{align*}

By using above we have

 \begin{align} \label{3.27}
 \int_{\{|u_k| \leq \delta\}} f(x,u_k)u_k \ dv_{g} &\leq \int_{\{|u_k| \leq \delta\}}
h(x,u_k)u_k(e^{\la u^2_k} - 1) \ dv_{g} \notag \\
                                               &\leq C\int_{\{|x| > \alpha\}} |u_k|^3
\ dv_{g} + O(\delta) \notag \\
                                               &\leq C\int_{\{|x| > \alpha\}}
(1 - |x|^2)^{\frac{3}{2}} \
dv_{g} + O(\delta) \notag \\
                                               &\leq C(1 - \alpha^2)^{\frac{1}{2}} + O(\delta),
\notag \\
                                               &= O(\delta).
\end{align}
Thus from \eqref{3.26} and \eqref{3.27} we get
\begin{align*}
 \int_{\bn} f(x,u_k) u_k \ dv_{g} = \int_{\{\delta \leq |u_k| \leq N\}} f(x,u_k) u_k \
dv_{g} + O(e^{-\frac{1}{2} \la \epsilon N^2}) + O(\delta).
\end{align*}

So the lemma follows by tending $k \rightarrow \infty$ and then tending $N
\rightarrow \infty, \delta \rightarrow 0$ successively.
\end{proof}

\begin{lem} \label{L3.7}
 Let $f(x,t) = h(x,t)(e^{\la t^2} - 1)$ be a function of critical growth. Fix $0 <
R_0 < 1$ and $0 < l_0 < R_0.$  Define 
\begin{align*}
 h_{0,l_{0}}(t) := \inf_{x \in B_{l_0}(0)} h(x,t), \ \mbox{and} \ M_0 := \sup_{t \geq 0} h_{0,l_{0}}(t)t
\end{align*}

and
\begin{align*}
k_0 = 
\begin{cases}
\frac{2}{M_0(R^2_0 - l^2_0)} \ \mbox{if} \ M_0 < + \infty, \\ 
0 \ \ \ \ \ \ \ \ \ \ \ \mbox{if} \ M_0 = + \infty .
\end{cases}
\end{align*}
Let $a \geq 0$ be such that 
\begin{align*}
 \sup_{||u|| \leq 1} \int_{\bn} f(x,au)u \ dv_{g} \leq a,
\end{align*}
if $\frac{k_0}{\la} < 1,$ then $a^2 < \frac{4 \pi}{\la}.$
\end{lem}

\begin{proof}
 From Lemma \ref{L3.2}, we have $a^2 \leq \frac{4 \pi}{\la}.$ Suppose if possible $a^2 = \frac{4 \pi}{\la}.$
 
Let $m_{l,R_{0}}$  be the Moser function defined in section 2. Then  $m_{l, R_{0}}$  is constant in
$\{|x| < l\}.$ Let $t = a m_{l,R_{0}}$ when $|x| < l,$ and
\begin{align*}
 1 \leq \frac{1}{(1 - |x|^2)^2} \leq \frac{1}{(1 - l^2)^2} \ \mbox{for} \ 0 < |x| < l,
\end{align*}
then we have
\begin{align} \label{3.28}
 a^2 &\geq \int_{\bn} f(x,a m_{l,R_{0}})(a m_{l,R_{0}}) \ dv_{g} \notag \\
     &\geq \int_{\{|x| < l\}} f(x,t)t \ dv_{g} \notag \\
     &\geq \int_{\{|x| < l\}} h(x,t)(e^{\la t^2} - 1)t \ dx, \notag \\
     &\geq 2 \pi h_{0,l_{0}}(t)t(e^{\la t^2} - 1)l^2. 
\end{align}
Now by our assumption $a^2 = \frac{4 \pi}{\la},$  using $e^{\la t^2} = \frac{R^2_0}{l^2}$ and
\eqref{3.28} we get
\begin{align*}
 \frac{4 \pi}{\la} &\geq 2 \pi h_{0,l_{0}}(t)t(R^2_0 - l^2)  \\
                   &\geq 2 \pi h_{0,l_{0}}(t)t(R^2_0 - l^2_0), \\
                   &\geq 2 \pi M_0 (R^2_0 - l^2_0).
\end{align*}
This gives $\la \leq \frac{2}{M_0(R^2_0 - l^2_0)} = k_0,$ a contradiction. Hence we must have $a^2 < \frac{4 \pi}{\la}.$ 

\end{proof}

Now we can prove Theorem \ref{PS}.

{\bf{Proof of Theorem \ref{PS}} :}
\begin{proof}
 Let $\{u_k\}$ be a sequence in $H^1_R(\bn)$ such that 
 \begin{align}
  \displaystyle{\lim_{k \rightarrow +\infty}} \ J_{\lambda}(u_k) &= c, \notag \\
  \displaystyle{\lim_{k \rightarrow +\infty}} \ J^{\prime}_{\lambda}(u_k) &= 0,    
 \end{align}
for some $c \in (0, \frac{2 \pi}{\la}).$\\

{\bf{Claim :}} $\{u_{k} \}$ is a bounded sequence in $H^1_R(\bn)$. \\

{\bf{Proof of Claim :}}
As $J_{\lambda}(u_k) \rightarrow c$ and $J^{\prime}_{\lambda}(u_k) \rightarrow 0, $ we have
\begin{align*}
 J_{\lambda}(u_k) \leq M_0 \ \mbox{and} \ \langle J^{\prime}_{\lambda}(u_k),u_k \rangle \leq M_0(1 + ||u_k||) 
\ \mbox{for all}\  k, 
\end{align*}
where $M_0 > 0$ is a constant. Also

\begin{align} \label{3.32}
 J_{\lambda}(u_k) - \frac{1}{2} \langle J^{\prime}_{\lambda}(u_k),u_k \rangle = I_{\lambda}(u_k).
\end{align} 
Which gives  $I_{\lambda}(u_k) \leq M_1(1 + ||u_k||)$ and hence form (ii) of Lemma \ref{L3.1}, we have, for $\epsilon >$
small,
\begin{align} \label{3.33}
 \int_{\bn} f(x,u_k)u_k \ dv_{g} \leq M_2(\epsilon)(1 + ||u_k||) + \epsilon ||u_k||^2.
\end{align}
Therefore from (C2) it follows that
\begin{align*}
 \int_{\bn} F(x,u_k) \ dv_{g} \leq C_0(\epsilon)(1 + ||u_k||) + \tilde C \epsilon ||u_k||^2.
\end{align*}
Now using the boundedness of $J_{\lambda}(u_k),$ and choosing $\epsilon$ such that $1 - \tilde C \epsilon > 0$ we have
\[
(1 - \tilde C \epsilon) ||u_k||^2 \leq C_0(1 + ||u_k||),
\]
and it proves  $||u_k||$'s are bounded. This proves the Claim. We also infer from \eqref{3.33} that

\begin{align} \label{3.34}
 \sup_k \int_{\bn} f(x,u_k)u_k \ dv_{g} < +\infty.
\end{align}

By extracting a subsequence (if necessary) we may assume that $u_k$ converges to a function $u \in H^1_R(\bn)$ weakly. Now we shall consider the following two cases:\\

{\bf{Case(i):}} \  $c \leq 0.$ \\
Using  \eqref{3.32} and (ii) of Lemma \ref{L3.1} we have,
 \begin{align*}
  0 \leq I_{\lambda}(u) &\leq \liminf_{k \rightarrow \infty}  I_{\lambda}(u_k) \\
              & = \liminf_{k \rightarrow \infty} \left\{J_{\lambda}(u_k) - \frac{1}{2} \langle J^{\prime}_{\lambda}(u_k), u_k\rangle\right\} \\
              &= c.
 \end{align*}
It follows that no Palais-Smale sequence exists if $c < 0.$ If $c = 0$ then from (ii) of Proposition \ref{L3.3} we have
\begin{align*}
 \lim_{k \rightarrow \infty} ||u_k||^2 = 
 2\lim_{k \rightarrow \infty} \left\{J_{\lambda}(u_k) + \int_{\bn} F(x,u_k) \ dv_{g} \right\} = 0,
 \end{align*}
and hence $u_k$ converges strongly to $0$ in $H^1_R(\bn).$ \\

{\bf{Case(ii)}} \  $c \in (0, \frac{2 \pi}{\la}).$ \\
First we shall show that $u \not\equiv 0.$ Suppose if possible $u \equiv 0.$
From \eqref{3.34} and (ii) of Proposition \ref{L3.3} we have
\begin{align*}
 \lim_{k \rightarrow +\infty} ||u_k||^2 &= 2 \lim_{k \rightarrow +\infty} \left\{ J_{\lambda}(u_k) + \int_{\bn} F(x,u_k) \ dv_{g} \right\} \\
                                        &=2c < \frac{4\pi}{\la}.
 \end{align*}
It follows that $u_k$ satisfies the hypothesis of Proposition \ref{L3.4} with $v_k = u_k, l = 2$  and hence we have
\begin{align*}
 \lim_{k \rightarrow +\infty} \int_{\bn} f(x,u_k)u_k \ dv_{g} = \int_{\bn} f(x,u)u \ dv_{g} = 0.
\end{align*}
This gives
\begin{align*}
 \lim_{k \rightarrow +\infty} I_{\lambda}(u_k) &= \lim_{k \rightarrow +\infty} 
 \left\{ \frac{1}{2} \int_{\bn} f(x,u_k)u_k \ dv_{g} - \int_{\bn} F(x,u_k) \ dv_{g} \right\} = 0.
\end{align*}
But then from \eqref{3.32} we get
\begin{align*}
 c = \lim_{k \rightarrow +\infty} J_{\lambda}(u_k) = \lim_{k \rightarrow +\infty} \left\{I_{\lambda}(u_k) + \frac{1}{2}\langle J^{\prime}_{\lambda}(u_k), u_k
 \rangle\right\} = 0, 
\end{align*}
which is a contradiction. Hence we must have $u \not\equiv 0.$ By definition of $J^{\prime}_{\lambda}(u)$ and standard density argument it follows
that
\begin{align}
 ||u||^2 = \int_{\bn} f(x,u)u \ dv_{g}.
\end{align}
Now since $u_k$ and $u$ satisfies all the hypothesis of Proposition \ref{L3.6} we have
\begin{align*}
 \lim_{k \rightarrow +\infty} \int_{\bn} f(x,u_k)u_k \ dv_{g} = \int_{\bn} f(x,u)u \ dv_{g} ,
\end{align*}
and hence by lower semi continuity of the norm we obtain
\begin{align*}
 ||u||^2 &\leq \liminf_{k \rightarrow +\infty} ||u_k||^2 ,\\
         &= 2\liminf_{k \rightarrow +\infty} \left\{J_{\lambda}(u_k) + \int_{\bn} F(x,u_k) \ dv_{g} \right\}, \\
         &= 2\liminf_{k \rightarrow +\infty} \left\{I_{\lambda}(u_k) + \frac{1}{2}\langle J^{\prime}_{\lambda}(u_k),u_k\rangle+ \int_{\bn} F(x,u_k) \ dv_{g})\right\},  \\
         &=  \liminf_{k \rightarrow +\infty} \left\{\int_{\bn} f(x,u_k)u_k \ dv_{g} + \langle J^{\prime}_{\lambda}(u_k),u_k\rangle \right\} ,\\
         &= \int_{\bn} f(x,u)u \ dv_{g}, \\
         &= ||u||^2 .
\end{align*}
This implies that $u_k \rightarrow u$ strongly in $H^1_R(\bn)$ and thus completes the proof of (i).

{\bf{Proof of part(ii)}.}

The proof goes in the same line as in \cite{ADI} with obvious modifications, so we will briefly sketch the proof here : \\
{\bf{Step 1:}} $\eta(f) > 0.$ \\

 If possible we assume \  $\eta(f) = 0,$ and let $\{u_k\}$ be a sequence in $\mathcal{M}$ such that $J_{\lambda}(u_k) = I_{\lambda}(u_k)$
converges to $0.$ Then from (ii) of Lemma \ref{L3.1} and proceeding as before, we can assume
\begin{align*}
  \displaystyle{\sup_k} \ ||u_k|| &< +\infty, \\
 \displaystyle{\sup_k} \ \int_{\bn} f(x,u_k)u_k \ dv_g &< +\infty. 
\end{align*}

By extracting a subsequence and using Fatou's lemma and Proposition \ref{L3.3} we can conclude that
\[ \label{100}
 u_k \rightarrow 0 \ \mbox{strongly in} \ H^1_R(\bn).
\]
Whereas considering $v_k = \frac{u_k}{||u_k||}$, we have $v_k$ converges weakly to $v$. Then by 
Proposition \ref{L3.4} and observing that $u_k \in \mathcal{M}$ we get
\begin{align*}
 1 &= \lim_{k \rightarrow \infty} \int_{\bn} \frac{f(x,u_k)}{u_k}v^2_k \ dv_g  \\
   &= \int_{\bn} f^{\prime}(x,0)v^2 \ dv_g = 0.
\end{align*}
This contradiction proves that $\eta(f) > 0.$ \\

Now for the second part we need the following  \\
{\bf{claim:}} For every $u \in H^1_R(\bn)\backslash \{0\}$, there exists a constant $\gamma(u) > 0$ such that 
$\gamma(u)u \in \mathcal{M}.$ In addition if one assume 
\[
 ||u||^2 \leq \int_{\bn} f(x,u)u \ dv_g,
\]
then $\gamma(u) \leq 1,$ and $\gamma(u) = 1$ iff $u \in \mathcal{M}.$ \\
Considering
\[
 \psi(\gamma) = \frac{1}{\gamma} \int_{\bn} f(x,\gamma u)u \ dv_g, \ \mbox{for} \ \gamma >0,
\]
one observes that 
\begin{align*}
\displaystyle{\lim_{\gamma \rightarrow 0}} \psi(\gamma) &= \int_{\bn} f^{\prime}(x,0)u^2 \ dv_g = 0 < ||u||^2, \\
\displaystyle{\lim_{\gamma \rightarrow \infty}} \psi(\gamma) &= \infty.
\end{align*}

So the first part of the claim follows by continuity of $\psi.$ Since by (C2), $\frac{f(x,tu)}{t}u$ is an 
increasing function of $t$, the second part of the claim follows.

{\bf{Step 2:}} $\eta(f)^2 < \frac{4\pi}{\la}.$ \\

In view of \eqref{PSlavel} and Lemma \ref{L3.7} it is enough to prove that
\[
 \displaystyle{\sup_{||u|| \leq 1}} \int_{\bn} f(x, \eta(f) u)u \ dv_g \leq \eta(f).
\]

Let $u \in H^1_R(\bn)$ with $||u|| = 1,$ by the above claim there exists $\gamma(u) > 0$ such that 
$\gamma(u)u \in \mathcal{M}.$
Then 
\[
 \frac{\eta(f)^2}{2} \leq J_{\lambda}(\gamma u) \leq \frac{\gamma^2}{2},
\]
that is $\eta(f) \leq \gamma,$ and hence by monotonicity of $\frac{f(x,tu)}{t}u$ with respect to $t$ we have
\[
\int_{\bn} \frac{f(x, \eta(f) u)}{\eta(f)} u \ dv_g \leq \int_{\bn} \frac{f(x, \gamma u)}{\gamma} u \ dv_g = 1,
\]
and this completes the proof.

\end{proof}

\end{section}

\section{Proof of main theorems}

 In this section we will prove the existence of solutions for the Eq. \eqref{main}. First we state the following
  abstract result :

\begin{lem}\label{abstract}

Let $f$ be a function of critical growth on $\bn.$  \\

\emph{1.} Let $u_{0} \in \mathcal{M}_{1}$ be such that $J^{\prime}_{\lambda}(u_{0}) \not\equiv 0.$ Then 
\[
 J_{\lambda}(u_{0}) > \inf \{ J_{\lambda}(u) : u \in \mathcal{M}_{1} \}.
\]
\ \ \ \emph{2.} Let $u_{1}$ and $u_{2}$ be two non negative linearly independent functions in $H^{1}_{R}(\bn).$ Then there exist 
$p, q \in \mathbb{R}$ such that $p u_{1} + q u_{2} \in \mathcal{M}_{1}.$
\end{lem}

The proof of above lemma follows from the result of Cerami-Solimini-Struwe (see \cite{CSS}) with obvious modifications.

\begin{rem}\label{remark1}
 Part (1) of above lemma holds for functions in $\mathcal{M}$ as well. 
\end{rem}

 {\bf{Proof of Theorem \ref{mt1} :}}

As $J_{\lambda}(u) = J_{\lambda}(|u|)$, it is enough to prove that the minimum is attained on
$\mathcal{M}$ for some nonzero function (thanks to above remark and principle of symmetric criticality). 
Hence we only need to show that there exists $u \in \mathcal{M}$ with $u \not\equiv
0$ such that
\begin{align*}
 J_{\lambda}(u) = \frac{\eta(f)^2}{2}, 
\end{align*}
and also  by part (ii) of Theorem \ref{PS}, we know 
\[
 0 <\eta(f)^2 < \frac{4\pi}{\la}.
\]

Let $\{u_k\}$ be a minimizing sequence. Since $J_{\lambda}=I_{\lambda}$ on $\mathcal{M},$
we have from (ii) of Lemma \ref{L3.1}
\begin{align} \label{4.1}
  \int_{\bn} f(x,u_k)u_k \ dv_g \leq C(1 + I_{\lambda}(u_k)) + \epsilon ||u_k||^2,
\end{align}
and arguing as before, we get
\begin{align} \label{4.2} 
 \sup_k ||u_k|| < +\infty , \\
 \label{4.2.}
 \sup_k \int_{\bn} f(x,u_k)u_k \ dv_g < +\infty.
\end{align}
By extracting a subsequence we can assume that $u_k$ converges to $u$ weakly in
$H^1_R(\bn)$.

{\bf{Claim:}} $u \not\equiv 0$ and $u \in \mathcal{M}.$ \\

{\bf{proof:}} If possible we assume \  $u \equiv 0.$ By \eqref{4.2.} and (ii) of
Proposition \ref{L3.3} we conclude that
\begin{align*}
 \lim_{k \rightarrow \infty} \int_{\bn} F(x,u_k) \ dv_g = 0.
\end{align*}
This gives
\begin{align} \label{4.3}
 \lim_{k \rightarrow \infty} ||u_k||^2 &= 2\lim_{k \rightarrow \infty} \left\{J_{\lambda}(u_k)
+ \int_{\bn} F(x,u_k) \ dv_g\right\} \notag \\
                                       &= \eta(f)^2 \in  (0, \frac{4\pi}{\la}).
\end{align}

Also \eqref{4.3} enables us to use Proposition \ref{L3.4} with $v_k = u_k$ and $l = 2$ to
conclude
\begin{align*}
 \lim_{k \rightarrow \infty} \int_{\bn} f(x,u_k)u_k \ dv_g = 0,
\end{align*}
which is not possible, otherwise this would give
\begin{align*}
 \eta(f)^2 = 2\lim_{k \rightarrow \infty} J_{\lambda}(u_k) = 2 \lim_{k \rightarrow \infty}
I_{\lambda}(u_k) = 0.
\end{align*}
Hence we must have $u \not\equiv 0.$ Now it remains to show that $u \in \mathcal{M}.$\\
First assume that 
\begin{align*}
 ||u||^2 > \int_{\bn} f(x,u)u \ dv_g.
\end{align*}
 This together with \eqref{4.2.} and Proposition \ref{L3.6} gives
\begin{align*}
 \lim_{k \rightarrow \infty} \int_{\bn} f(x,u_k)u_k \ dv_g = \int_{\bn} f(x,u)u \ dv_g.
\end{align*}
Then lower semi continuity of the norm implies
\begin{align*}
 ||u||^2 &\leq \liminf_{k \rightarrow \infty} \ ||u_k||^2 \\
         &=    \liminf_{k \rightarrow \infty} \int_{\bn} f(x,u_k)u_k \ dv_g \\
         &= \int_{\bn} f(x,u)u \ dv_g.
\end{align*}
But this contradicts our initial assumption.  Hence we must have $||u||^2 \leq \int_{\bn} f(x,u)u \
dv_g.$ It follows from the proof of part (ii) of Theorem \ref{PS} 
that there exists $0 < \gamma \leq 1$ such that $\gamma u \in \mathcal{M}.$
Then by monotonicity of $\frac{f(x,tu)u}{t}$
we have, 
\begin{align*}
 \frac{\eta(f)^2}{2} \leq J_{\lambda}(\gamma u) &= I_{\lambda}(\gamma u) \\
                                          &\leq I_{\lambda}(u)     \\
                                          &\leq \liminf_{k \rightarrow \infty}
I_{\lambda}(u_k) \\
                                          &= \liminf_{k \rightarrow \infty} J_{\lambda}(u_k) =
\frac{\eta(f)^2}{2},
\end{align*}
and then again using part (ii) of Theorem \ref{PS}, we conclude that $\gamma = 1$ and $J_{\lambda}(u) =
\frac{\eta(f)^2}{2}.$ This completes the proof.\\

Our next job is to investigate existence of sign changing solution whose  
 proof will heavily depend on the following concentration lemma.
The proof of concentration lemma follows in the same lines as in (\cite{ADI1}, Lemma 3.1) with some modifications (See Appendix 1).\\
Here we state the lemma:
\begin{lem}\label{mainlemma}
 Let $f(x,t) = h(x,t) (e^{\la t^2} - 1)$ be a function of critical growth on $\bn$ and $V$ be the one dimensional subspace
defined by $\{pu_{0} : p \in \mathbb{R} \}$
of $H^{1}_{R}(\bn).$ Let $h_{0,\beta}(t) = \inf \{h(x,t); x \in \overline{B(0, \beta)} \} $ 
and $C(V) = \sup \{J_{\la}(u) : u \in V \}.$ Assume that :\\
%
%
 For every $ N > 0,$ there exists  $t_{N} > 0$ such that 
\[
 h_{0,\beta}(t)t \geq e^{N t}, \ \forall \ t \geq t_{N}.
\]

Then there exists $\epsilon_{0} > 0$ such that, for $0 < \epsilon < \epsilon_{0}, $

\begin{equation}\label{c0}
\sup_{u \in V, t\in \mathbb{R}} J_{\lambda} (u + t m_{\epsilon, \beta}) < C(V) + \frac{2\pi}{\la}, 
\end{equation}\label{mnl}
where $m_{\epsilon, \beta}$ is the Moser function.
\end{lem}

Now we can prove Theorem \ref{mt2}.\\
{\bf{Proof of Theorem \ref{mt2}}}\\
From Lemma \ref{abstract}, it is sufficient to show that the infimum of $J_{\lambda}$
is achieved on $\mathcal{M}_{1}.$ We first claim that: \\
{\bf{Claim 1.}} $0 < \frac{\eta_{1}(f)^2}{2} < \frac{\eta(f)^2}{2} + \frac{2\pi}{\lambda}.$ \\
From definition it is clear that  $\eta_{1}(f)^2 \geq \eta(f)^2. $ By Theorem \ref{mt1}, let 
$u_{0} \in \mathcal{M}$ be such that 
\begin{equation}\label{11.2}
\sup_{\alpha \in \mathbb{R}} J_{\lambda}(\alpha u_{0}) = J_{\lambda}(u_{0}) = \frac{\eta(f)^2}{2} > 0, 
\end{equation}
 hence this gives $\eta_{1}(f) > 0.$ From (2) of Lemma \ref{abstract}, for any $n_{0} > 0,$
\begin{equation}\label{21.2}
\frac{\eta_{1}(f)^2}{2} \leq \sup_{p,q \in \mathbb{R}} J_{\lambda}(pu_{0} + q m_{n_{0},\beta}), 
\end{equation}
where $m_{n_{0},\beta}$ is the Moser function. Again from \eqref{11.2} and by considering $V= \{ p u_{0}, p \in \mathbb{R} \}$
in  Lemma \ref{mainlemma}, there exists $n_{1} > 0$ such that for $0 < n_{0} <  n_{1}, $
\begin{equation}\label{31.2}
\sup_{p,q \in \mathbb{R}} J_{\lambda} (p u_{0} + q m_{n_{0}, \beta}) < \frac{\eta(f)^2}{2} + \frac{2 \pi}{b}. 
\end{equation}
 Hence Claim 1 follows from \eqref{21.2} and \eqref{31.2}.\\

Let $u_{k}$ be in $\mathcal{M}_{1}$ such that 
\[
 \lim_{k \rightarrow \infty}J_{\lambda}(u_{k}) = \frac{\eta_{1}(f)^2}{2}.
\]
Since $ J_{\lambda} = I_{\lambda}$ on $\mathcal{M}_{1},$ hence from part (ii) of  Lemma \ref{L3.1}, we obtain 
\begin{equation}\label{41.2}
 \sup_{k}||u_{k}|| < \infty, \ \ \sup_{k} \int_{\bn} f(x,u_{k}) u_{k} \ dv_{g} < \infty.
\end{equation}
Therefore we can extract a subsequence of $\{ u_{k}\}$ such that
\[
 u_{k}^{\pm} \rightarrow u_{0}^{\pm} \ \mbox{weakly}.
\]
 From \eqref{41.2} and Proposition \ref{L3.3}, we get 
\begin{equation}\label{51.2}
\lim_{k \rightarrow \infty} \int_{\bn} F(x, u_{k}^{\pm}) \ dv_{g} = \int_{\bn} F(x , u_{0}^{\pm}) \ dv_{g}. 
\end{equation}

In accordance to Claim 1, we can choose $\epsilon > 0,$ $m_{0} > 0$ such that for all $ k \geq m_{0},$
\[
 \eta_{1}(f)^2 \leq 2 J_{\lambda}(u_{k}) \leq \eta(f)^2 + \frac{4 \pi}{\lambda} - \epsilon,
\]
 this together with $J_{\lambda}(u_{k}^{\pm}) \geq \frac{\eta(f)^2}{2}$ gives 
\begin{equation}\label{61.2}
 J_{\lambda}(u_{k}^{\pm}) \leq \frac{2 \pi}{\lambda} - \frac{\epsilon}{2}.
\end{equation}
                                                              
{\bf{Claim 2.}} $u_{0}^{\pm} \not\equiv 0$ and $||u_{0}^{\pm}||^2 \leq \int_{\bn} f(x , u_{0}^{\pm})u_{0}^{\pm} \ dv_{g}.$\\

We shall only prove this for $u_{0}^{+}.$ A similar proof will holds for $u_{0}^{-}$ as well. Suppose $u_{0}^{+} \equiv 0.$ Then 
from \eqref{51.2} and \eqref{61.2}, we have 

\[
 \limsup_{k \rightarrow \infty} ||u^{+}_{k}||^2 = 2 \limsup_{k \rightarrow \infty} 
\left (J_{\lambda}(u_{k}^{+}) + \int_{\bn} F(x, u_{k}^{+}) \ dv_{g} \right ) \leq \frac{4 \pi}{\lambda} - \epsilon.
\]

Therefore from Proposition \ref{L3.4},
\begin{equation}\label{71.2}
\lim_{k \rightarrow \infty} \int_{\bn} f(x, u_{k}^{+}) u_{k}^{+} \ dv_{g} = 0.
 \end{equation}
 Since $u_{k}^{+} \in \mathcal{M},$ we get from \eqref{71.2}, 
 $\lim_{k \rightarrow \infty} ||u^{+}_{k}|| = 0.$ Which together with $\eta(f) > 0,$ gives a contradiction.
This proves $u^{+}_{0} \not\equiv 0.$ Now suppose 
\begin{equation}\label{81.2}
||u^{+}_{0}||^2 > \int_{\bn} f(x, u_{0}^{+}) u_{0}^{+} \ dv_{g}.
\end{equation}
 Then $\{u_{k}^{+}, u_{0}^{+} \}$ satisfies all the hypothesis of Proposition \ref{L3.6}, and hence 
\[
 \lim_{k \rightarrow \infty} \int_{\bn} f(x, u_{k}^{+})u_{k}^{+} \ dv_{g} = \int_{\bn} f(x, u_{0}^{+})u_{0}^{+} \ dv_{g}.
\]
Therefore we have 
\[
 ||u^{+}_{0}||^2 \leq \liminf_{k \rightarrow \infty} ||u_{k}^{+}||^2 
= \liminf_{k \rightarrow \infty} \int_{\bn} f(x, u_{k}^{+})u_{k}^{+} \ dv_{g} = \int_{\bn} f(x, u_{0}^{+})u_{0}^{+} \ dv_{g},
\]
which contradicts \eqref{81.2} and hence this proves Claim 2.\\

Thanks to Claim 2, the property
\[
 ||u_{0}^{\pm}||^2 \leq \int_{\bn} f(x, u_{0}^{\pm})u_{0}^{\pm} \ dv_{g},
\]
 enables us to choose $0 < r_{1} \leq 1,$ $0 < r_{2} \leq 1,$ such that 
\[
 v = r_{1}u_{0}^{+} - r_{2}u_{0}^{-} \in \mathcal{M}_{1}.
\]
Also we have 
\begin{align*}
 \frac{\eta_{1}(f)^2}{2} \leq J_{\lambda}(v) & 
\leq I_{\lambda}(v) = I_{\lambda}(r_{1}u_{0}^{+}) + I_{\lambda}(r_{2}u_{0}^{-}) \\
& \leq I_{\lambda}(u_{0}^{+}) + I_{\lambda}(u_{0}^{-}) \leq \liminf_{k \rightarrow \infty} I_{\lambda}(u_{k}) \\
& = \lim_{k \rightarrow \infty} J_{\lambda}(u_{k}) = \frac{\eta_{1}(f)^2}{2}. 
\end{align*}
Hence  $r_{1} = r_{2} = 1,$ which gives $u_{0} \in \mathcal{M}_{1}$ 
and $J_{\lambda}(u_{0}) = \frac{\eta_{1}(f)^2}{2}.$ This completes the proof of Theorem \ref{mt2}.\\

{\bf{Proof of Theorem \ref{mt3} :}}

 The proof of this theorem has the similar lines of proof of ( \cite{ADI1}, Theorem 1.3) with an account
  of lemma similar to ( \cite{ZN} Lemma 3.1). For the sake of brevity, we have omitted the detailed verification.

\section{Appendix 1}
In this section we shall try to give a sketch of the proof of Lemma \ref{mainlemma}. \\
{\bf{Proof of Lemma \ref{mainlemma} :}}\\
From the radial estimate \ref{radialestimate}, it is very clear that blow up can occur only at the origin. 
Hence we need to analyze 
only near origin. Denote the one dimensional vector space $\{ p u_{0} : p \in \mathbb{R} \}$ by $V.$\\
Let $ u_{l} = p_{l} u_{0} + t_{l} m_{l,\beta}$ be such that $t_{l} \geq 0 $ and 
\[
 J_{\lambda} (u_{l}) = \displaystyle{\sup_{ \alpha , t \in \mathbb{R}}} J_{\lambda} (\alpha u_{0} + t m_{l,\beta}).
\]
 
Since $J_{\lambda}^{\prime}(u_{l}) = 0$ on $\{ \alpha u_{0} + t m_{l, \beta} : \alpha, t \in \mathbb{R}\},$ hence 
\begin{equation}\label{c1}
 ||u_{l}||^{2} = \int_{\bn} f(x , u_{l})u_{l} dv_{g}.
\end{equation}
Now suppose that \eqref{c0} is not true. Then there exists a sequence $l_{n}$ such that 
$l_{n} \rightarrow 0$ as $ n \rightarrow 0$ and for $ v_{n} := \alpha_{n}u_{0} = \alpha_{l_{n}} u_{0},$ 
$m_{n,\beta} = m_{l_{n}, \beta},$ $t_{n} = t_{l_{n}},$  $u_{n} = u_{l_{n}},$
\begin{equation}\label{c2}
C(V) + \frac{2 \pi}{\lambda} \leq J_{\lambda}(u_{n}).
 \end{equation}

{\bf{Step 1.}} $\{ v_{n}\}$ and $t_{n}$ are bounded.\\
Suppose it is not true. Then either 
\[
 \displaystyle{\lim_{n \rightarrow \infty}} \frac{t_{n}}{||v_{n}||} > 0, \ \mbox{or} \ 
\displaystyle{\lim_{n \rightarrow \infty}} \frac{t_{n}}{||v_{n}||} = 0. 
\]
In first case, there exist a subsequence of $\{ v_{n},t_{n} \}$ and a constant $C > 0$ such that for large $n,$
\begin{equation}\label{c31}
\frac{t_{n}}{||v_{n}||} \geq C \ \mbox{and} \ t_{n} \rightarrow \infty \ \mbox{as} \ n \rightarrow \infty.
\end{equation}
As $||m_{n, \beta}|| = 1,$ we have from \eqref{c31}, 
\begin{equation}\label{c41}
||u_{n}||^2 = t_{n}^2 + 2 t_{n} \langle v_{n}, m_{n, \beta} \rangle + ||v_{n}||^2 \leq C_{1} t^{2}_{n},
\end{equation}
where $C_{1} = 1 + \frac{2}{C} + \frac{1}{C^2}.$ Since $\frac{||v_{n}||}{t_{n}}$ is bounded 
and $v_{n} \in \{ pu_{0} : p \in \mathbb{R} \},$ therefore  $\frac{|v_{n}|_{\infty}}{t_{n}}$ is bounded. Hence 
for $ x \in B(0,l_{n})$ and for large $n,$
\begin{align}\label{c51}
 u_{n}(x) & = v_{n}(x) + t_{n} m_{n, \beta}(x) \\ \notag
& = t_{n}m_{n, \beta}(x) \left( 1 + \frac{v_{n}(x)}{t_{n}} \frac{1}{m_{n, \beta}(x)} \right ) \\ \notag
& \geq \frac{1}{2} t_{n} m_{n, \beta }(x). 
\end{align}
Hence we have 
\begin{align}
 C_{1}t_{n}^2 \geq ||u_{n}||^2  & = \int_{\bn} f(x, u_{n})u_{n} \ dv_{g} \\ \notag 
& \geq \int_{B(0, \beta)} h(x,u_{n})u_{n}(e^{\lambda u_{n}^2} - 1) \ dv_{g} \\ \notag
& \geq \int_{B(0, \beta)} h_{0,\beta}(u_{n}) u_{n}(e^{\lambda u_{n}^2} - 1) \ dv_{g} \\ \notag
& \geq \int_{B(0, l_{n})} h_{0,\beta}(u_{n}) u_{n}(e^{\lambda u_{n}^2} - 1) \ dv_{g} \\ \notag
& \geq C_{2} (e^{\frac{\lambda}{8} t_{n}^2 m_{n,\beta}^{2}(0)} - 1) l_{n}^2,
\end{align}
where $C_{2}$ is a positive constant. This implies that 
\[
 C_{1} \geq C_{2} (e^{{\frac{\lambda t_{n}^2}{16 \pi}} \log \frac{\beta}{l_{n}} - 2 \log \frac{1}{l_{n}} - 2 \log t_{n}} - l_{n}^2) 
\rightarrow \infty,
\]
as $ l_{n} \rightarrow 0,$ which gives a contradiction and hence first one can not occur.\\
In second case, first note that $||v_{n}|| \rightarrow \infty.$ Let 
\[
 z_{n} = \frac{v_{n}}{||v_{n}||}, \ \epsilon_{n} = \frac{t^{2}_{n}}{||v_{n}||^2} + 
\frac{2 t_{n}}{||v_{n}||} \langle z_{n},m_{n, \beta} \rangle.
\]

Then upto a subsequence and using the fact that $ z_{n} \in \{p u_{0} : p \in \mathbb{R} \},$ we can assume 
\begin{equation}\label{c61}
 \displaystyle{\lim_{n \rightarrow \infty}} z_{n} = z_{0}, \ z_{0} \in \{p u_{0} : p \in \mathbb{R} \} \setminus \{ 0 \},
\displaystyle{\lim_{n \rightarrow \infty}} \epsilon_{n} = 0.
\end{equation}
Also
\begin{align}
 ||u_{n}||^2  & = ||v_{n}||^2 + 2 t_{n} \langle v_{n}, m_{n, \beta } \rangle + t_{n}^2 \\ \notag
 & = ||v_{n}||^2 (1 + \epsilon_{n}). \notag\\
\end{align}
Hence
\begin{equation}\label{c71}
 \frac{u_{n}}{||u_{n}||} = \frac{1}{( 1 + \epsilon_{n})^\frac{1}{2}} \left( z_{n} + \frac{t_{n}}{||v_{n}||} m_{n, \beta} \right) 
\rightarrow z_{0} \not\equiv 0 \ \mbox{in} \ H^{1}(\bn).
\end{equation}
Now using Fatou's lemma,
\begin{align}
 \infty & = \int_{\bn} \displaystyle{\liminf_{n \rightarrow \infty}} \
\frac{f(x,u_{n})}{u_{n}} \left ( \frac{u_{n}}{||u_{n}||} \right)^2 dv_{g} \\ \notag 
 & \leq \displaystyle{\liminf_{n \rightarrow \infty}} \ \frac{1}{||u_{n}||^2} \int_{\bn} f(x, u_{n}) u_{n} \ dv_{g} = 1,
\end{align}
which is a contradiction. Hence this proves Step 1.\\

Therefore upto a subsequence we can assume
\[
 \displaystyle{\lim_{n \rightarrow \infty}} v_{n} = v_{0} \ \mbox{in}\ V, 
\displaystyle{\lim_{n \rightarrow \infty}} t_{n} = t_{0}. 
\]

Also $u_{n} \rightharpoonup  v_{0}$ weakly in $H^{1}(\bn)$ and for almost all $x$ in $\bn.$
\begin{rem}
 $\displaystyle{\lim_{n \rightarrow \infty}} v_{n} = v_{0} \ \mbox{in}\ V$ implies there exist a 
sequence $\alpha_{n} \in \mathbb{R}$ such that
$\alpha_{n}u_{0} \rightarrow \alpha u_{0}.$
\end{rem}
Using Proposition \ref{L3.3}, we conclude 
\begin{equation}\label{c81}
 \displaystyle{\lim_{n \rightarrow \infty}} \int_{\bn} F(x, u_{n}) \ dv_{g} = \int_{\bn} F(x, v_{0}) \ dv_{g}.
\end{equation}

Now letting $ n \rightarrow \infty$ in \eqref{c2} and using convergence results, we get
\begin{equation}\label{st1m}
 C(V) + \frac{2 \pi}{\lambda} \leq J_{\lambda}(v_{0}) + \frac{t_{0}^2}{2} \leq C(V) + \frac{t_{0}^2}{2}.
\end{equation}

{\bf{Step 2.}} $t_{0}^2 = \frac{4 \pi}{\lambda}$ and $J_{\lambda}(v_{0}) = C(V).$\\

From \eqref{st1m} we have $t_{0}^2 \geq \frac{4 \pi}{\lambda}.$ Suppose $t_{0}^2 > \frac{4 \pi}{\lambda},$ then
arguing same as in step 2 of (\cite{ADI1}, Lemma 3.3) and using Step 1, we can get, for $ n \geq n_{0},$
 \begin{align}\label{st21}
  \mbox{M} = \displaystyle{\sup_{n}} \ ||u_{n}||^2 \geq C_{1} [l_{n}^{-2 ((1 + \frac{\epsilon}{4})(1 - \epsilon_{n}) - 1)} - l_{n}^2], 
 \end{align}
for some positive constant $C_{1}.$ As $\epsilon_{n} \rightarrow 0, l_{n} \rightarrow 0,$ \eqref{st21} gives a 
contradiction. Hence $t_{0}^2 = \frac{4 \pi}{\lambda}$ and \eqref{st1m} gives $J_{\lambda}(v_{0}) = C(V).$ \\

{\bf{Step 3.}} There exist positive constants $n_{0}$ and $C_{0}$ such that for all $ n \geq n_{0},$ 

\begin{equation}\label{st34}
 C_{0} + \log \left(M + \pi p_{n, \beta}(0) l_{n}^2 \right) \geq \left(t_{n}^2 - \frac{4 \pi}{\lambda} \right) m_{n, \beta}^2(0) 
- \frac{1}{\lambda} \epsilon_{n}m_{n,\beta}(0) + \frac{1}{\lambda} \log p_{n,\beta}(0),
\end{equation}
where \[
       \epsilon_{n} = 2 \lambda t_{n} \displaystyle{\sup_{x \in \bn}} \ |v_{n}(x)|, 
      \]
and 
\[
 p_{n,\beta}(0) = \inf \left \{ t h_{0,\beta}(t): 
t \in\left[ \frac{1}{2} t_{n}m_{n,\beta}(0), 2 t_{n}m_{n,\beta}(0) \right] \right \}.
\]
A straightforward calculations gives,
\begin{align}\label{st35}
 M \geq \pi p_{n,\beta}(0) l_{n}^2 (e^{\lambda t_{n}^2 m_{n, \beta}^2(0) - \epsilon_{n}m_{n, \beta}(0) } - 1),
\end{align}
and from \eqref{st35}, Step 3 follows easily.\\

{\bf{Step 4.}} There exists a constant $C_{1} > 0$ such that for large $n,$ 
\begin{equation}\label{st46}
 \left(\log \ \frac{\beta}{l_{n}} \right)^{\frac{1}{2}} 
\left(\frac{4 \pi}{\lambda} - t_{n}^2 \right) \leq \tilde C_{1}|\Delta v_{n}|_{L^{2}_{loc}} \leq C_{1} \alpha_{n}.
\end{equation}

Proof of Step 4 follows by convexity of $ t \rightarrow F(x,t)$ and elliptic regularity.\\
Finally we are in position to prove the final step. By hypothesis, given any $ N  > 0$ and a compact set $\overline{B(0, \beta)},$
 there exists $t_{N,\beta} > 0$ such that $h_{0,\beta}(t) t \geq e^{Nt}$ for all $t \geq t_{N, \beta}.$ Since we  
$m_{n, \beta}(0) \rightarrow \infty,$ by \eqref{st34} and \eqref{st46} we obtain for large $n,$
\begin{equation}\label{st51}
 \left[\frac{Nt_{n}}{2 \lambda} - \frac{\epsilon_{n}}{\lambda}- 
\frac{\sqrt{2\pi}C_{0}}{\left( \log \ \frac{\beta}{l_{n}}\right)^{\frac{1}{2}}} - 
\frac{\sqrt{2 \pi} \log \left(M + \pi e^{Nt_{n}} l_{n}^2 \right)}{\left(\log \ \frac{\beta}{l_{n}} \right)^{\frac{1}{2}}} \right]
\leq  \frac{C_{1}}{\sqrt{2 \pi}} \alpha_{n}.
\end{equation}
Since $\epsilon_{n},$ $\alpha_{n}$ are bounded and $t_{n} \rightarrow t_{0} > 0.$ From above, we get 
\[
 \frac{Nt_{0}}{2 \lambda} \leq \tilde{C_{1}}
\]
 for some positive constant $\tilde{C_{1}}.$ Since $N$ is arbitrary, we get a contradiction. Hence this proves the lemma.

\section{Appendix 2}
This section is devoted to the existence of non radial solutions. Typically existence of non radial solutions on the hyperbolic space is a difficult question due to the lack of compactness through 
vanishing (earlier mentioned). We have made an attempt to give an existence theorem for non radial solution in certain cases. We have eliminated concentration at infinity by considering
a suitable growth condition on the nonlinearity or in other words, by considering
a penalty assumption which sets the asymptotic nonlinear part to become zero.

Moreover, by invariance with respect to M\"obius transformations we are able to prove 
P.L.Lions Lemma \ref{L3.5} for $H^{1}(\mathbb{B}^{N})$  which plays an important role in the subsequent proof. 
In this regard we first modify the function of critical growth. \\

A model problem of our study is :
\begin{equation} \label{nonradial equation}
 -\Delta_g u = f(x,u), \ \ \ x \in \bn,
\end{equation}
where $f(x,t)$ is a function of critical growth as defined in Definition \ref{1}
 and in addition it satisfies some growth condition near infinity. To be precise we assume $f(x,t) = h(x,t)(e^{\la t^2} - 1)$
satisfying : \\

$(\overline{C3})$
\begin{align} \label{C4 assumption}
   F(x,t) \leq C(g(x) + f(x,t)),\ g \in L^1(\bn, dv_g) \cap L^p(\bn, dv_g), \ \ \mbox{for some} \ p \in (1,2], 
\end{align}

$(C5)$ There exists a $\delta > 0$ such that
\begin{align}
 \frac{h(x,t)}{(1 - |x|^2)^{\delta}} \in L^{\infty}(\{|x| > \alpha\} \times [-N, N]), \ \ \mbox{for all} \ N.
\end{align}

$(C6)$ for every $\epsilon>0,$ there exists $\alpha(\epsilon) > 0,$ such that
\begin{align} \label{aditional assumption}
h(x,t) \geq (1 - |x|^2)^l e^{-\epsilon t^2} \ \mbox{for some}  \  l > 0, \ |x| > \alpha(\epsilon) \ \mbox{and} \ t \ 
\mbox{positive large}. 
\end{align}

A prototype example of such function is $f(x,t) = (1 - |x|^2)^l t (e^{\la t^2} - 1)$ for some $ l > 0.$ Hence unlike in the radial case, here  we can allow maximum singularity of order 
$\frac{1}{(1 - |x|^2)^{2 - \epsilon}}$ at the boundary. We prove

\begin{thm} \label{main theorem for nonradial}
 Let $f(x,t)$ be a function of critical growth satisfying $(\overline{C3}), (C5)$ and $(C6).$ Further assume that
 \begin{align}
  \lim_{t \rightarrow \infty} \inf_{x \in K} h(x,t)t = \infty,
 \end{align}
for every compact set $K \subset \bn,$  then \eqref{nonradial equation} has 
a positive solution.
\end{thm}
 
 \begin{rem}
 The function of critical growth $f(x,t)$ defined in the above theorem is not necessarily a 
 radial function in its first variable. Thus from invariance of  $\Delta_{\mathbb{B}^{N}}$ under 
 orthogonal transformation we infer that the solution thus obtained in Theorem \ref{main theorem for nonradial} 
 is non radial if we assume $f(x,t)$ non radial in its first variable.
 \end{rem}

It is easy to see that under the above assumptions  
we can estimate the growth of $F(x,u)$ and $f(x,u)$ near infinity. As a consequence we can prove 
all the necessary tools to acquire existence of solution of \eqref{nonradial equation}. For the 
shake of completeness we will outline some of the steps. In the rest of the section $f(x,t)$ stands for
function of critical growth satisfying $(\overline{C3}), (C5), (C6)$ and $I_{\lambda} , J_{\la}$ are as 
defined before (see section $3$ ) corresponding to this $f.$\\

Now we  will provide a  proof of lemmas which significantly different from the radial ones. 
We have used some refined arguments and taking the advantage of growth conditions to prove the following lemmas.  

\begin{lem}
 Let $\{u_k\}$ be a sequence in $H^1(\bn)$ converging weakly to a function $u$ in $H^1(\bn).$ Further 
 assume that $\sup_k \int_{\bn} f(x,u_k)u_k \ dv_g < +\infty,$
 then 
 \begin{align}
\int_{\bn} F(x,u_k) \ dv_g \rightarrow \int_{\bn} F(x,u) \ dv_g.   
 \end{align}
\end{lem}
\begin{proof}
 It is enough to show that $\int_{\{|x| > \alpha \}} F(x,u_k) \ dv_g$ can be made arbitrarily small
 by choosing $\alpha$ close to $1.$ Indeed
 \begin{align*}
  \int_{\{|x| > \alpha\}} F(x,u_k) \ dv_g &\leq C \int_{\{|x| > \alpha\} \cap \{|u_k| \leq N \} } 
  (1 - |x|^2)^{\delta} (e^{(\la + \epsilon) u_k^2} - 1) \ dv_g 
  + \int_{\{ |x| > \alpha \} \cap \{|u_k| > N \}} F(x,u_k) \ dv_g \\ 
  &\leq Ce^{(\la + \epsilon) N^2} (1 - \alpha)^{\delta} ||u_k||^2 + 
  \frac{C}{N} \int_{\bn} (g(x)u_k + f(x,u_k)u_k) \ dv_g \\
  &\leq C e^{(\la + \epsilon)N^2}(1 - \alpha)^{\delta} + \frac{C}{N}.
 \end{align*} 
Here we used $g \in L^p(\bn, dv_g)$ for some $p \in (1,2],$ and this completes the proof.
\end{proof}
\begin{lem} \label{nonradial L2}
 Given $\mu > 0, $ there exists a constant $C(\mu) > 0,$ such that
 \begin{align} \label{nonradial bound using I}
  \int_{\bn} f(x,u)u \ dv_g \leq C(\mu) (1 + I_{\lambda}(u)) + \mu ||u||^2, \ \ \ \mbox{for all} \ u \in H^1(\bn).
 \end{align}
\end{lem}
\begin{proof}
 We note that,
 \begin{align} \label{bound using I}
  \int_{\{ u \leq 4M_1\}} f(x,u) \ dv_g &\leq \int_{\{ u \leq 4M_1\} \cap \{|x| > \alpha\}} f(x,u) \ dv_g + C(\alpha) \notag \\
  &\leq C (1 - \alpha)^{\delta} \int_{\bn} u^2 \ dv_g + C(\alpha) \notag \\
  &\leq C (1 - \alpha)^{\delta} ||u||^2 + C(\alpha) = \frac{\mu}{2}||u||^2 + C(\mu),
 \end{align}
by choosing $\alpha$ close to $1.$ Therefore proceeding as in Lemma \ref{L3.1} and using \eqref{bound using I} we 
get \eqref{nonradial bound using I}.
\end{proof}

\begin{lem} \label{nonradial L3}
 Let $f(x,t) = h(x,t)(e^{\la t^2} - 1)$ be a function of critical growth, then
 \begin{align}
  d^2 := \sup \{ c^2 : \sup_{u \in H^1(\bn), ||u|| \leq 1} \int_{\bn} f(x,cu) u \ dv_g < +\infty \} = \frac{4\pi}{\la}
 \end{align}
\end{lem}
\begin{proof}
 The proof goes in the same line as before with obvious modifications. We will only mention the steps which 
 differ from the previous one. We see that,
 \begin{align*}
  \int_{\{|x| > \alpha \}} f(x, cu) u \ dv_g & \geq C \int_{\{|x| > \alpha \} \cap \{u \geq t _0 \}} 
  h(x,cu)(e^{\la c^2 u^2} - 1) \ dv_g \\  
  &\geq C \int_{\{|x| > \alpha \} \cap \{u \geq t_0\}} ( 1 - |x|^2)^{l-2} (e^{\la(1 - \epsilon) c^2 u^2} - 1) \ dx. 
 \end{align*}
From this and proceeding as in Lemma \ref{L3.3} we conclude that,
\begin{align*}
\sup_{u \in H^1(\bn), ||u||^2 \leq 1} \int_{\bn} (1 - |x|^2)^{l-2} (e^{\la(1 - \epsilon) c^2 u^2} - 1) \ dx < +\infty,
\end{align*}
and hence $\la d^2 \leq 4\pi.$ The reverse inequality is similar as in Lemma \ref{L3.3}. 
 \end{proof}
 
 \begin{lem}
  Let $\{u_k\}$ and $\{v_k\}$ be bounded sequences in $H^1(\bn)$ converging weakly to $u$ and
$v$ respectively. Further assume that $\sup_k ||u_k||^2  < \frac{4\pi}{\la},$
then for all $l \geq 2,$
\begin{align} \label{condition PS}
 \lim_{k \rightarrow \infty} \int_{\bn} \frac{f(x,u_k)}{u_k} v_k^l \ dv_g = \int_{\bn} \frac{f(x,u)}{u} v^l \ dv_g.
\end{align}
\end{lem}

\begin{proof}
 As before we can show that
 \begin{align*}
  \int_{\{|u_k| > N\}} \frac{f(x,u_k)}{u_k} v_k^l \ dv_g = O\left(\frac{1}{N}\right).
 \end{align*}
Now we estimate $\int_{\{|u_k| \leq N\} \cap \{|x| > \alpha \}} \frac{f(x,u_k)}{u_k} v_k^l \ dv_g.$
In fact we can show that,
\begin{align*}
 \int_{\{|u_k| \leq N\} \cap \{|x| > \alpha \}} \frac{f(x,u_k)}{u_k} v_k^l \ dv_g
 \leq C(N)e^{(\la + \epsilon) N^2} (1 - \alpha)^{\delta},
\end{align*}
so that,
\begin{align} \label{convergence for PS condition}
 \int_{\bn} \frac{f(x,u_k)}{u_k} v_k^l \ dv_g = 
 \int_{\{ |u_k| \leq N\} \cap \{|x| \leq \alpha \}} \frac{f(x,u_k)}{u_k} v_k^l \ dv_g + O\left(\frac{1}{N}\right)
 + C(N) e^{(\la + \epsilon) N^2} (1 - \alpha)^{\delta}.
\end{align}
From \eqref{convergence for PS condition} we can easily see that \eqref{condition PS} holds.
\end{proof}
A similar argument gives,
\begin{lem}
 Let $\{u_k\}$ be a sequence in $H^1(\bn)$ converging weakly to a non-zero function $u$
and assume that : 

(i) there exists $c \in (0 , \frac{2\pi}{\la})$ such that $J_{\la} (u_k) \rightarrow c,$ 

(ii) $||u||^2 \geq \int_{\bn} f(x,u)u \ dv_g,$ 

(iii) $\sup_k \int_{\bn} f(x,u_k) u_k \ dv_g < +\infty.$ Then
\begin{align}
 \lim_{k \rightarrow \infty} \int_{\bn} f(x,u_k)u_k \ dv_g = \int_{\bn} f(x,u)u \ dv_g. 
\end{align}
\end{lem}

\textbf{Proof of Theorem \ref{main theorem for nonradial}:}\\
We omit the proof because it goes in the same line as in Theorem \ref{mt1} (see section $4$ and section $5$ 
for details).

\textbf{Acknowledgement:} The authors are grateful to Prof. Sandeep for useful discussions during the preparation of the paper. 
The first author is partially supported by the Research Project FIR (Futuro in Ricerca) 2013 \emph{Geometrical and qualitative aspects of PDE's}.
We would like to thank the referee for his precious comments regarding the improvements
of our result in the radial case and his inquisitiveness to know about non radial situations.


\vspace{2.00cm}
(Debdip Ganguly)\\
{\textit E-mail address:} debdip@math.tifrbng.res.in

\vspace{0.50cm}
(Debabrata karmakar)\\
{\textit E-mail address:} debkar@math.tifrbng.res.in


\begin{thebibliography}{99}



\bibitem{ADI} Adimurthi;
 \emph{Existence of positive solutions of the semilinear Dirichlet problem with critical growth for the $N-$Laplacian},
Ann.Scuola Norm.Sup.Pisa Cl.Sci (4) 17, (1990), pp. 393-413. 

\bibitem{ADI1} Adimurthi, S.L.Yadava; \emph{Multiplicity results for semilinear elliptic equations in a bounded domain of 
$\mathbb{R}^{2}$ involving critical exponents}, Ann.Scuola Norm.Sup.Pisa Cl.Sci (4) 17, (1990), pp. 481-504. 

\bibitem{ADI2} Adimurthi, S.L.Yadava; \emph{A note on non-existence of nodal solutions of the semilinear elliptic equations
with critical exponents in $\mathbb{R}^2$ }, Trans. Amer. Math. Soc., 332 (1992) 1, pp. 449-458. 

\bibitem{ADIT} Adimurthi, K. Tintarev;  \emph{On a version of Trudinger-Moser inequality with M\"obius shift invariance},
Calc. Var. Partial Differential Equations 39 (2010), no. 1-2, pp. 203-212.



\bibitem{AYS} Adimurthi, S.L.Yadava, P.N.Srikanth; \emph{Phenomena of critical exponent in $\mathbb{R}^2$}, 
Proc. Roy. Soc. Edinburgh Sect. A 119 (1991), no. 1-2, pp. 19-25. 

 \bibitem{APR} Adimurthi, S.Prashanth; \emph{Critical exponent problem in $\mathbb{R}^{2}$-border-line 
between existence and non-existence of positive solutions for 
Dirichlet problem}, Adv. Differential Equations 5 (2000), no. 1-3, pp. 67-95.

\bibitem{AP} F.V.Atkinson, L.A.Peletier; \emph{Elliptic equation with critical growth}, Math.Inst.Univ.Leiden, Rep 21 (1986).


 \bibitem{B} W. Beckner; \emph{On the Grushin operator and hyperbolic symmetry}, 
 Proc. Amer. Math. Soc. Vol 129 (2001), pp. 1233-1246.
   

\bibitem{EBG} E.Berchio, A.Ferrero, G.Grillo; \emph{Stability and qualitative 
properties of radial solutions of the Lane-Emden-Fowler equation on Riemannian models},
 J. Math. Pures Appl. (9) 102 (2014), no. 1, pp. 1-35.


\bibitem{BGG} M.Bonforte, F.Gazzola,  G.Grillo and J. Luis V\'{a}zquez ; 
\emph{Classification of radial solutions to the Emden \emph{-} Fowler equation on the hyperbolic space}. 
 Calc. Var. Partial Differential Equations 46 (2013), no. 1-2, pp. 375-401.



 \bibitem{PS} M.Bhakta, K.Sandeep;  \emph{Poincar$\acute{e}$- Sobolev equations
 in the hyperbolic space}, 
 Calc. Var. Partial Differential Equations 44 (2012), no. 1-2, pp. 247-269. 

\bibitem{LM} L.Battaglia, G.Mancini; \emph{Remarks on the Moser Trudinger inequality}, 
Adv. Nonlinear Anal. 2 (2013), no. 4,  pp. 389-425.

\bibitem{HS1} D.Castorina,  I.Fabbri, G.Mancini, K.Sandeep; \emph{Hardy- Sobolev
 inequalities and hyperbolic symmetry},  Atti Accad. Naz. Lincei Cl. Sci. Fis. Mat. Natur. Tend. Lincei (9) 
Mat. Appl. 19 (2008), pp. 189-197.
 
 \bibitem{HS} D.Castorina,  I.Fabbri, G.Mancini, K.Sandeep; \emph{Hardy- Sobolev
 extremals,hyperbolic symmetry and 
 scalar curvature equations}, Journal of Differential Equations 246 (2009),  pp. 1187-1206.

\bibitem{CC1} L.Carleson, S.Y.A. Chang; \emph{On the existence of an extremal function for an inequality of J. Moser},
Bull. Sci. Math. 110, (1986), pp. 113-127.


\bibitem{CSS} G.Cerami, S.Solimini, M.Struwe; \emph{Some existence results for superlinear elliptic boundary 
 value problems involving critical exponents}, J.Funct. Anal. 69 (1986), pp. 289-306.


\bibitem{DR1} de Figueiredo, D.G., do O, Joao Marcos,  J.M., Ruf, B; \emph{On an inequality by N. Trudinger and J. Moser and related
elliptic equations}, Commun. Pure Appl. Math. 55(2), (2002)  pp. 135-152.




\bibitem{DR3} de Figueiredo, D.G., do  O, Joao Marcos, J.M., Ruf, B; \emph{Critical and subcritical elliptic systems in dimension two},
Indiana Univ. Math. J. 53(4), (2004), pp. 1037-1054. 

 
\bibitem{DS} D.Ganguly, K.Sandeep; \emph{Sign changing solutions of the Brezis-Nirenberg problem in the Hyperbolic space},
Calc. Var. Partial Differential Equations 50 (2014), pp. 69-91.


\bibitem{DS1} D.Ganguly, K.Sandeep; \emph{ Nondegeneracy of positive solutions of semilinear elliptic 
problem in the hyperbolic space},  Commun. Contemp. Math. 17 (2015), no. 1, 1450019, 13 pp.


\bibitem{PLL} P.L.Lions; \emph{The Concentration Compactness principle in the calculus of variations, part-I},
Revista mathematica Iberoamericana, No.1 (1985), pp. 185-201.



\bibitem{pl1} P.L.Lions; \emph{The concentration compactness principle in the calculus of variation. The limit case. I},
 Ann.Inst. H. Poincar\'{e} 1 (1984), pp. 223-283.
 

\bibitem{MS1} G.Mancini, K.Sandeep; \emph{Moser-Trudinger inequality on
conformal discs},  Commun. Contemp. Math. 12 (2010), no. 6,  pp. 1055-1068. 


\bibitem{MS2} G.Mancini, K.Sandeep; \emph{Extremals for Sobolev and Moser inequalities in hyperbolic space},
 Milan J. Math. 79 (2011), no. 1,  pp. 273-283. 


\bibitem{TMS} G.Mancini, K. Tintarev, K.Sandeep; \emph{Trudinger-Moser inequality in the
 hyperbolic space $\mathbb{H}^{N}$},  Adv. Nonlinear Anal. 2 (2013), no. 3, pp. 309-324. 



\bibitem{MS} G.Mancini, K.Sandeep; \emph{On a semilinear elliptic
equation in $H^n$}, Ann. Sc. Norm. Super. Pisa Cl. Sci. (5), {\bf 7}, no. 4, (2008),  pp.  635-671. 




\bibitem{JM} J. Moser, \emph{A sharp form of an inequality by N.Trudinger}, Indiana Univ.
Math. J. 20 (1970/71), pp. 1077-1092.


\bibitem{LLG}  L.Nguyen, L.Guozhen; \emph{Elliptic equations and systems 
with subcritical and critical exponential growth without the Ambrosetti-Rabinowitz condition}.
 J. Geom. Anal. 24 (2014), no. 1,  pp. 118-143.




 
\bibitem{ZN} Z.Nehari; \emph{Characteristic values associated with a class of nonlinear second-order differential equations},
Acta Mathematica, 105 (1961), pp. 141-175.


\bibitem{PR} R.Palais; \emph{The principle of symmetric criticality}, Commun. Math. Phys. 69, (1979), pp.19-30.

\bibitem{PANDA} R.Panda; \emph{On semilinear Neumann 
problems with critical growth for the n-Laplacian}, Nonlinear Anal. 26 (1996), no. 8, pp. 1347-1366.

\bibitem{PANDA1} R.Panda; \emph{Nontrivial solution of a quasilinear elliptic equation with critical growth in $\mathbb{R}^{N},$}
Proc. Indian Acad. Sci. Math. Sci. 105 (1995), no. 4, pp. 425-444.


\bibitem{JR} John.G. Ratcliffe; \emph{Foundations of Hyperbolic Manifolds}, Graduate
Texts in Mathematics, Vol-149, Springer.



\bibitem{TINU} K.Tintarev; \emph{Is the Trudinger-Moser nonlinearity a true critical nonlinearity?}
Discrete Contin. Dyn. Syst. 2011, Dynamical systems, differential equations and applications. 
8th AIMS Conference. Suppl. Vol. II, 1378-1384.






 \end{thebibliography}
\end{document}